\documentclass{amsart}
\usepackage{amssymb}
\usepackage{amsmath}
\usepackage{amsfonts}
\usepackage{mathtools}
\usepackage{amsthm}
\usepackage{bm}
\usepackage{graphicx}
\usepackage{subcaption}
\usepackage{youngtab}
\usepackage[margin=1.25in]{geometry}

\newcommand{\EMD}{\text{EMD}}
\newcommand{\GL}{\text{GL}}

\newcommand{\bmu}{\bm{\mu}}

\newtheorem{theorem}{Theorem}
\newtheorem{corollary}[theorem]{Corollary}
\newtheorem{lemma}[theorem]{Lemma}
\newtheorem{prop}[theorem]{Proposition}
\theoremstyle{definition}
\newtheorem*{definition}{Definition}
\newtheorem*{example}{Example}
\newtheorem*{remark}{Remark}

\begin{document}

\title[A generalization for the expected value of EMD]{A generalization for the expected value of the earth mover's distance}

\author{William Q. Erickson}
\address{3200 North Cramer Street, Milwaukee, WI 53211}
\email{wqe@uwm.edu}


\dedicatory{}

\begin{abstract}

The earth mover's distance (EMD), also called the first Wasserstein distance, can be naturally extended to compare arbitrarily many probability distributions, rather than only two, on the set $[n]=\{1,\dots,n\}$.  We present the details for this generalization, along with a highly efficient algorithm inspired by combinatorics; it turns out that in the special case of three distributions, the EMD is half the sum of the pairwise EMD's.  Extending the methods of Bourn and Willenbring (2020), we compute the expected value of this generalized EMD on random $d$-tuples of distributions, using a generating function which coincides with the Hilbert series of the Segre embedding. We then use the EMD to analyze a real-world data set of grade distributions.
\end{abstract}

\maketitle

\section{Introduction}
\label{s:intro}

We generalize a result appearing in \cite{bw}, in which the authors compute the expected value of the earth mover's distance (EMD) between two probability distributions, by means of a generating function. The EMD can be viewed as the solution to a problem in transport theory, first considered in \cite{monge} by French geometer Gaspard Monge in 1781.  (Although the term ``earth mover" seems to have been coined only in the 1990s, it is pointed out in Villani's monumental reference \cite{villani} that the title of Monge's original treatise translates, more or less, as ``On the theory of material extracted from the earth and input to a new construction."  Monge, then, truly was the original earth mover.) Nearly 200 years later, in \cite{hoffman}, Monge's name was given to a critical property of certain cost arrays for which his problem can be solved by a greedy algorithm. Throughout the late 1980s and 1990s, in \cite{a&p} and \cite{bein}, this Monge property was generalized to higher-dimensional arrays.  (See also \cite{kline} for a more recent treatment.) It is an essential fact in this paper (whose proof is reserved for Section \ref{s:mongeproof} at the end) that the $d$-dimensional cost array associated with our $\EMD_d$ has this Monge property. 

Section \ref{s:EMD2}, written for those readers unfamiliar with the EMD, presents a simple example and points out all the relevant details which will reappear in our generalization.

We begin in Section \ref{s:EMDd} by defining an earth mover's ``distance" $\EMD_d$ between $d$ distributions; the classical EMD treated in \cite{bw} coincides with $\EMD_2$.  We actually find that on three distributions, $\EMD_3$ equals half the sum of the the three $\EMD_2$ values, although no such relationship holds for $d>3$.  

In Section \ref{s:discreet}, en route to constructing a generating function, we define a discrete version of $\EMD_d$ which compares histograms instead of probability distributions, and we describe an efficient computational method using a generalization of the RSK correspondence from combinatorics.  

In Section \ref{s:expval}, we encode the values of the discrete $\EMD_d$ in a generating function, which we manipulate in order to extract the expected value.  Translating this discrete result back into the continuous setting, we prove the main theorem of this paper (Theorem \ref{expval}), which is a recursive formula to compute the expected value of $\EMD_d$.  We then apply our theory in Section \ref{s:realworld} to analyze a real-world data set of grade distributions. 

Finally, in Section \ref{s:repthy}, we mention a connection between our generating function and the Segre embedding in algebraic geometry.  We also exhibit a certain infinite-dimensional representation of the Lie algebra $\mathfrak{su}(p,q)$, whose action corresponds to manipulating the distributions compared by our EMD.

Since the appearance of \cite{bw}, the problem of finding the expected value of $\EMD_2$ has been solved from an analytical approach in \cite{volkmer}. The setup has also been specialized in \cite{kretschmann} to a data set of distributions with a fixed average value.

We believe that the result in this paper --- a method to evaluate the ``closeness" of arbitrarily many distributions --- has great potential as a tool in data analysis.  In evaluating teaching and assessment practices at the university level, for instance, we can now assign a single value to an entire course by evaluating the EMD between the individual sections, and then track the behavior of that course's EMD for different groups of instructors, different course coordinators, fall vs.\ spring semesters, and other variables.  We can even assign EMD values to individual exams and other assessments using the grade distributions in various sections; or in the other direction, we can compare different courses to each other, both within and outside a given department.  In all of these settings, we believe that the generalized EMD can contribute to an interesting cluster analysis of the kind proposed in \cite{bw}.

\textbf{Acknowledgements:} The author would like to thank Rebecca Bourn and Jeb Willenbring, the authors of \cite{bw}, for the conversations about their original paper.  Jeb's observations about the connections to representation theory were especially vital to Section \ref{s:repthy}.

\section{EMD between 2 distributions: summary and an example}
\label{s:EMD2}

For readers unfamiliar with the classical EMD, we summarize the idea here.  Consider two probability distributions on the finite set of integers $[n]=\{1,\dots,n\}$.  (More vividly, in place of a ``probability distribution," imagine $n$ bins of earth whose combined mass is one unit, located at $1,\dots,n$ on the number line.) Intuitively, the EMD between the two distributions measures the ``cheapest" cost of moving earth between the bins so as to equalize the distributions, where the ``cost" of moving one unit of earth is the distance of the move.  For example, the cost of moving $0.25$ units of earth from bin $2$ to bin $5$ is $0.25\cdot (5-2)=0.75$.  To make this precise, we define the \textbf{cost function} $C:[n]\times [n]\longrightarrow \mathbb Z_{\geq 0}$, where $C(i,j)$ is the cost of moving one unit of earth from bin $i$ to bin $j$.  In this case, clearly $C(i,j)=|i-j|$.  

Any solution which equalizes the two distributions --- whether or not it is the optimal solution --- can be encoded in an $n\times n$ matrix $J$.  Necessarily, the row sums of $J$ will correspond to the first distribution, and the column sums to the second, so the entries of $J$ must sum to $1$.  

We present a brief example to show how the entries of $J$ give (possibly ambiguous, but equivalent) step-by-step instructions to equalize the two distributions.  The procedure we give here is not the most direct (see Section 2 of \cite{bw}), but it will provide the best intuition when we generalize to $d$ distributions in the next section. The less-than-rigorous descriptions below will be formalized in the next section in terms of the taxicab metric.

\begin{example}

Consider the two distributions $\mu_1=(0.3,\: 0.3, \:0.4)$ and $\mu_2=(0.1, \: 0, \: 0.9)$.  Hence $n=3$.  Then one matrix (among infinitely many) with the prescribed row and column sums is
\begin{equation*}
    J=\begin{bmatrix}
    .1 & 0 & .2\\
    \phantom{.}0 & 0 & .3\\
    \phantom{.}0 & 0 & .4
    \end{bmatrix}_{\textstyle .}
\end{equation*}
The nonzero entries of $J$ correspond to moving earth as follows:
\begin{itemize}
    \item $J_{1,1}=0.1$.  Note that the coordinates $(1,1)$ are already equal to each other, so we do not have to move the $0.1$ units of earth at all.
    \item $J_{1,3}=0.2.$  Now the coordinates $(1,3)$ are not equal; in order to make them equal with as little cost as possible, we have three valid options, all of which have cost 2:
    \begin{itemize}
        \item In the first coordinate, we could add $2$ to make the change $1 \rightarrow 3$.  This corresponds to moving the $0.2$ units of earth in $\mu_1$, from bin 1 to bin 3.
        \item In the second coordinate, we could subtract $2$ to make the change $3\rightarrow 1$.  This corresponds to moving the $0.2$ units of earth in $\mu_2$, from bin 3 to bin 1.
        \item We could add $1$ to the first coordinate $(1\rightarrow 2)$ and subtract 1 from the second coordinate $(3 \rightarrow 2)$.  This corresponds to moving $0.2$ units of earth in $\mu_1$ from bin 1 to bin 2, and then moving $0.2$ units of earth in $\mu_2$ from bin 3 to bin 2.
    \end{itemize}
    \item $J_{2,3}=0.3$.  The cheapest ways to equalize the coordinates $(2,3)$ are the following two options, each with cost 1:
    \begin{itemize}
        \item In the first coordinate, we could add $1$ to make the change $2\rightarrow 3$.  This corresponds to moving the $0.3$ units of earth in $\mu_1$, from bin 2 to bin 3.
        \item In the second coordinate, we could subtract $1$ to make the change $3\rightarrow 2$.  This corresponds to moving the $0.3$ units of earth in $\mu_2$, from bin 3 to bin 2.
    \end{itemize}
    \item $J_{3,3}=0.4$.  Since the coordinates $(3,3)$ are already equal, we do not have to move the $0.4$ units of earth at all.
\end{itemize}
\end{example}

Now, depending upon which of the above options we choose at each step, this process can result in any of six distinct pairs of final distributions $\mu'_1$ and $\mu'_2$.  But within each possible pair, as the reader can check, we always finish with $\mu'_1=\mu'_2$, as desired.  Furthermore, the total cost of all the earth moved is independent of our choices, since all options above minimized the cost at each step.  (Also note that the cost at each step was always equal to $|i-j|$, coinciding with the cost function $C$ we defined earlier.) In this case, the total cost of the earth moved was 
\begin{equation*}
    0.1(0)+0.2(2)+0.3(1)+0.4(0)=\mathbf{0.7}.
\end{equation*}  

The $\EMD$ between $\mu_1$ and $\mu_2$ is, by definition, the infimum (actually the minimum) of the set of total costs, taken over all possible matrices $J$ with the prescribed row and column sums.  In this example, although not obvious at first glance, $0.7$ is in fact the least possible cost, and so $\EMD(\mu_1,
\mu_2)=0.7$.  This turns out to be a consequence of the fact that the support of $J$ lies in a \textbf{chain}: in other words, if we put the product order $\preceq$ on $[n]\times[n]$, we see that
\begin{equation*}
(1,1) \preceq (1,3) \preceq (2,3) \preceq (3,3);
\end{equation*} 
this pairwise comparability is what we mean by a chain in $[n]\times [n]$.  This fact --- that support in a chain implies minimality --- is equivalent, on a deeper level (see \cite{hoffman}), to the fact that our cost function $C$, if considered as an $n\times n$ array, has the ``Monge property" alluded to in the introduction; in this case, the greedy algorithm to solve the earth mover's problem (known as the ``northwest corner rule"; see \cite{bein})  eliminates one row or column at each step, meaning the support of the solution matrix $J$ is always a chain.  

There is one phenomenon here in the $d=2$ case which will \emph{not} generalize to $d>2$: in the above example, we could have removed any ambiguity by deciding that we would move earth within $\mu_1$ exclusively, so that both final distributions would equal $\mu_2$.  Therefore, we could interpret the problem as finding the cheapest way to transport material from a ``source" or ``supply vector" ($\mu_1$) to a ``sink" or ``demand vector" ($\mu_2$).  For $d>2$, however, the optimal solution at each step may require moving earth in any or all of the distributions, and so we lose the binary supply-demand interpretation of the problem.  

Having presented the big picture, without details, in the $d=2$ case, we now proceed to build up the general case for arbitrary $d$.  Throughout the next section, the reader can verify that the definitions and results coincide with those found in this simple example where $d=2$.

\section{Extending EMD to $d$ distributions}
\label{s:EMDd}

\subsection{Definitions and notation}

Let $\mathcal P_n$ denote the set of probability distributions on $[n]$. Assume the uniform probability measure on the $d$-fold product $\mathcal P_n \times \cdots \times \mathcal P_n$, defined by its embedding into $\mathbb R^{dn}$.  Our goal is to compare an arbitrary number of elements of $\mathcal P_n$, written as the $d$-tuple $\bmu:=(\mu_1,\dots,\mu_d)$.  We should keep in mind that each $\mu_i$ is itself an $n$-tuple whose components sum to 1.  Throughout this paper, we write the sum of a vector's components using absolute value bars, so in this case, $|\mu_i|=1$.  We will denote the $k^\text{th}$ component of $\mu_i$ by $\mu_i(k)$, which is just the value of the distribution $\mu_i$ at $k\in [n]$.  To each $\bmu$ there corresponds the set $\mathcal J_{\bmu}$ of joint distribution arrays, defined as follows.  

For an array $J$, we will write $J(m_1,\dots,m_d)$ for the entry at position $(m_1,\dots,m_d)$.  Now, we define $\mathcal J_{\bmu}$ as the set containing all those arrays $J\in \mathbb R_{\geq 0}^{n\times \cdots \times n}$ whose sums within the coordinate hyperplanes coincide with $\bmu$.  Specifically, fixing $m_i=k$, we must have
\begin{equation}
\label{sums}
    \sum_{m_1,\dots,\widehat m_i,\dots,m_d=1}^n J(m_1,\dots,\underbrace{k}_{m_i},\dots,m_d)=\mu_i(k).
\end{equation}
In other words, summing all the entries whose positions in the array have $k$ as their $i^\text{th}$ coordinate, we obtain the $k^\text{th}$ component of $\mu_i$.  In the familiar $d=2$ case, $i=1$ gives us the row sums, and $i=2$ the column sums.  For $d=3$, see Figure \ref{fig:cube} for an illustration.

\begin{figure}[h]
 
\begin{subfigure}[t]{0.32\textwidth}
\includegraphics[width=0.9\linewidth]{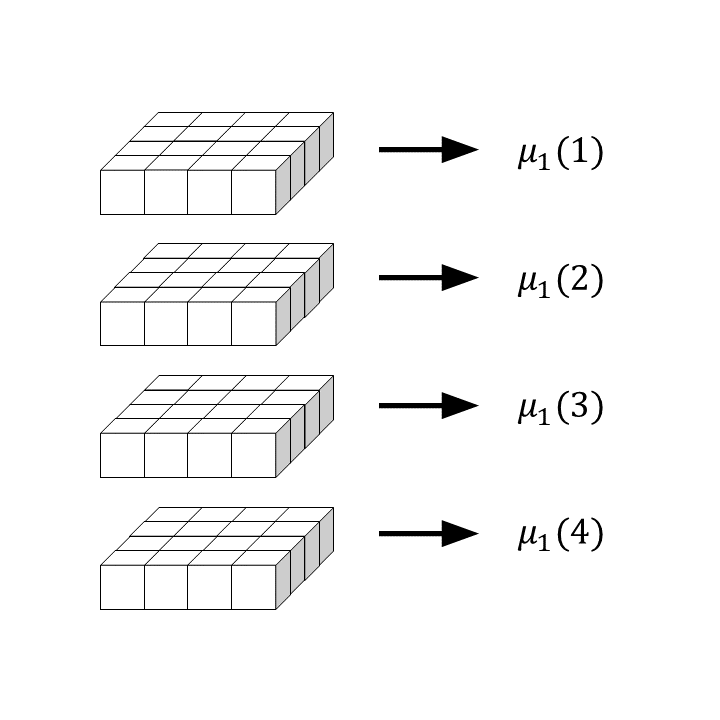} 
\end{subfigure}
\hfill
\begin{subfigure}[t]{0.32\textwidth}
\includegraphics[width=0.9\linewidth]{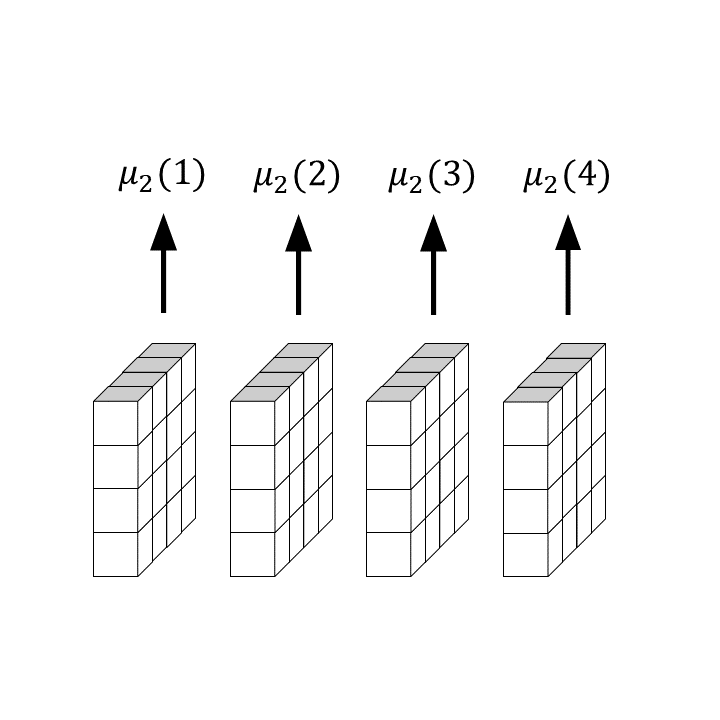}
\end{subfigure}
\hfill
\begin{subfigure}[t]{0.32\textwidth}
\includegraphics[width=0.9\linewidth]{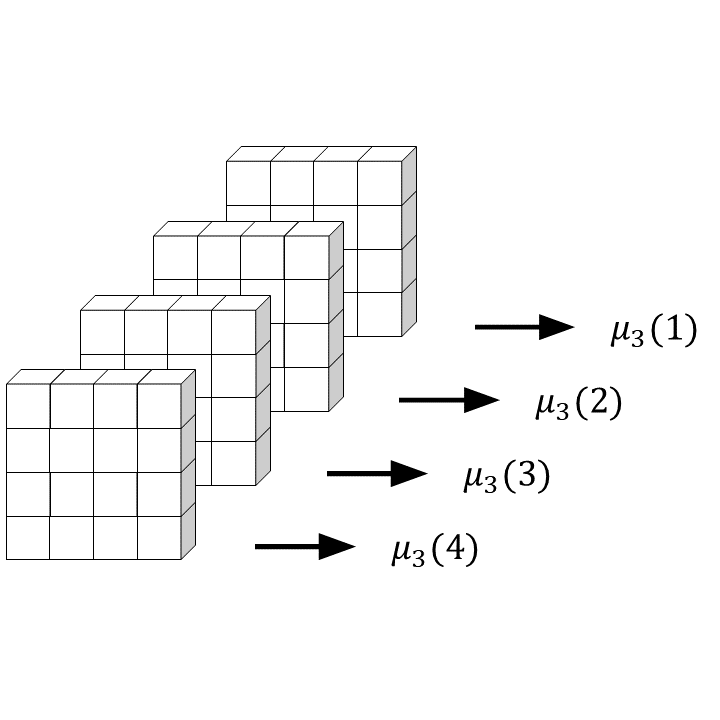}
\end{subfigure}
 
\caption{An illustration of the conditions in equation \eqref{sums}, in the case where $d=3$ and $n=4$. Given some $\bmu=(\mu_1,\mu_2,\mu_3)$, every array in $\mathcal J_{\bmu}$ satisfies the above relations, where each arrow represents the sum of the entries in the designated plane.}
\label{fig:cube}
\end{figure}

Any array $J\in \mathcal J_{\bmu}$ can be thought of as a solution to the earth mover's problem for $n$ bins, determined by the distributions in $\bmu$.  This means we need a $d$-dimensional analog of the ``cost" function from Section 2, and the natural candidate arises from the taxicab metric on $[n]^d$.  Specifically, for each position in a $d$-dimensional array, we want the associated cost equal the taxicab distance to the main diagonal, i.e., to the nearest position in the array whose coordinates are all equal.  (This ``equality of coordinates" property of the main diagonal, as we recall from Section 2, corresponded to zero earth being moved.)  Roughly speaking, this cost is the fewest number of $\pm1$'s we need to add in order to equalize all the coordinates.  For example, the most efficient way to equalize the coordinates of the position $(5,4,5,5,5,7,5)$ is to add $1$ to the $4$, and then to subtract $2$ from the $7$, for a total cost of $3$.  This is precisely the taxicab distance to the main diagonal, specifically to the position $(5,5,5,5,5,5,5)$.  Just as in the example from the previous section, this distance-finding exercise corresponds to moving earth:
\begin{itemize}
    \item When we added $1$ to the $2^\text{nd}$ coordinate to make the change $4\rightarrow 5$, we moved a unit of earth in the $2^\text{nd}$ distibution $\mu_2$ from bin $4$ to bin $5$.
    \item When we subtracted $2$ from the $6^\text{th}$ coordinate to make the change $7\rightarrow 5$, we moved a unit of earth in the $6^\text{th}$ distibution $\mu_6$ from bin $7$ to bin $5$.
\end{itemize}

\begin{example} 
Consider the three distributions
\begin{align*}
    \mu_1&=(0.5,\:0.1,\:0.4),\\
    \mu_2&=(0.5,\:0.2,\:0.3),\\
    \mu_3&=(0.7,\:0.2,\:0.1).
\end{align*}
Then one array in $\mathcal J_{\bmu}$ is, for instance,
\begin{equation}
\label{arrayexample}
    J=\begin{bmatrix}
    [.5,0,0]&[\phantom{.}0,0,0]&[0,\phantom{.}0,\phantom{.}0]\\
    [\phantom{.}0,0,0] & [.1,0,0] & [0,\phantom{.}0,\phantom{.}0]\\
    [\phantom{.}0,0,0] & [.1,0,0] & [0,.2,.1]
    \end{bmatrix}_{\textstyle ,}
\end{equation}
flattened so that the first coordinate specifies the row, the second coordinate specifies one of the three main columns, and the third coordinate specifies the position inside the triple at that position.  The nonzero entries are
\begin{align*}
    J(1,1,1)&=0.5\\
    J(2,2,1)&=0.1\\
    J(3,2,1)&=0.1\\
    J(3,3,2)&=0.2\\
    J(3,3,3)&=0.1.
\end{align*}
This information tells us how to arrive at the solution corresponding to $J$:
\begin{itemize}
    \item The cost of $(1,1,1)$ is $0$ since it is already on the main diagonal, so we do not move the $0.5$ at all. 
    \item The cost of $(2,2,1)$ is $1$, since in the $3^\text{rd}$ coordinate we must make the change $1\rightarrow 2$.  This means that in the $3^\text{rd}$ distribution $\mu_3$, we move $0.1$ from bin 1 to bin 2.  Currently $\mu'_3=(0.6,\:0.3,\:0.1)$.
    \item The cost of $(3,2,1)$ is $2$, since we equalize the coordinates most efficiently by subtracting $1$ from the $1^\text{st}$ coordinate $(3\rightarrow 2)$ and adding $1$ to the $3^\text{rd}$ coordinate $(1\rightarrow 2)$.  Hence, we move $0.1$ from bin 3 to bin 2 in $\mu_1$, and from bin 1 to bin 2 in $\mu_3$.  Now $\mu'_1=(0.5,\:0.2,\:0.3)$ and $\mu'_3=(0.5,\:0.4,\:0.1)$.
    \item The cost of $(3,3,2)$ is $1$, by adding 1 to the $3^\text{rd}$ coordinate.  This corresponds to moving $0.2$ from bin 2 to bin 3 in $\mu_3$.  Now $\mu'_3 = (0.5,\:0.2,\:0.3)$.
    \item The cost of $(3,3,3)$ is $0$, so we do not move the $0.1$ anywhere.
\end{itemize}
Note that our final result is that all three distributions are the same, as desired: $\mu'_1=\mu'_2=\mu'_3=(0.5, \: 0.2, \: 0.3)$.  Also note that we rigged this example, unlike that in Section 2, so that none of the steps would present more than one optimal option, although in general there certainly might exist several different solutions for the same array $J$.  But of course in each case the total cost is the same. 

The natural computation now is to find that total cost, by multiplying the amount of earth moved at each step by the number of bins it was moved; in other words, multiply each entry in $J$ by the cost of its position, then add these products together:
\begin{equation*}
    0.5(0)+0.1(1)+0.1(2)+0.2(1)+0.1(0)=\mathbf{0.5}.
\end{equation*} This completes the example.
\end{example}  

Of course, there is no guarantee that this is the least costly way to equalize the three distributions; this is simply the solution corresponding to one particular array $J$, and a different array in $\mathcal J_{\bmu}$ might give a different total cost.  When we finally define our generalized EMD, it will be defined as the \emph{least} possible cost for any $J \in \mathcal J_{\bmu}$.  First, however, we should record a formula for the cost of an array position, to improve upon the somewhat sloppy method by inspection we have used so far.

The formula for the $d$-dimensional taxicab distance from a point to a line is derived in \cite{taxicab}.  In our case, the line of interest is the main diagonal, which passes through $(1,\dots,1)$ in the direction $\langle 1,\dots,1\rangle$.  This distance, and therefore our cost function $C$, turns out to be 
\begin{equation}
\label{cost}
    C(m_1,\dots,m_d)=\min_{i \in [d]} \left\{ \sum_{j\neq i}|m_i-m_j|\right \}.
\end{equation}
This cost function $C$ can also naturally be thought of as an $n \times \cdots \times n$ array, so we will occasionally refer to the ``cost array" in this paper. 

There is also a more direct way to compute $C$, which will be convenient later.  Let $\mathbf m := (m_1,\dots,m_d)$, and let $\widetilde{\mathbf m}$ denote the vector whose components are those of $\mathbf m$ rearranged in ascending order; e.g., if $\mathbf m = (7,4,5,3,1)$, then $\widetilde{\mathbf m}=(1,3,4,5,7)$.

\begin{prop} 
\label{costalt}
Equation \eqref{cost} can be computed as     $\displaystyle C(\mathbf m)=\sum_{i=1}^{\lfloor d/2 \rfloor}
    \widetilde m_{d-i+1}-\widetilde m_i.$
\end{prop}

As an example before the proof, take $\mathbf m = (7,4,5,3,1)$ as above.  Then by the proposition, to compute $C(\mathbf m)$, we instead look at $\widetilde{\mathbf m}$ and sum up the pairwise differences working outside-in:
\begin{align*}
\widetilde{\mathbf m}&=(\overbrace{1,\overbrace{3,4,5}^{5-3=\mathbf{2}},7}^{7-1=\mathbf{6}}),\\
\text{therefore }C(\mathbf m)&=6+2\\
&=8.
\end{align*}

\begin{proof}
For fixed $i\in[d]$, we have 
\begin{align*}
\sum_{j\neq i}|m_i-m_j|&=(\widetilde m_2 - \widetilde m_1) + 2(\widetilde m_3 - \widetilde m_2) + 3(\widetilde m_4-\widetilde m_3)+\cdots +(i-1)(\widetilde m_i - \widetilde m_{i-1})\\
&+(\widetilde m_d - \widetilde m_{d-1})+2(\widetilde m_{d-1}-\widetilde m_{d-2})+3(\widetilde m_{d-2}-\widetilde m_{d-3})+\cdots +(d-i)(\widetilde m_{i+1}-\widetilde m_i),
\end{align*}
which is minimized when $i=\lfloor \frac{d+1}{2} \rfloor$.  Making this evaluation in the displayed sum, we find that the sum telescopes; when $d$ is even, we obtain
$$
-\widetilde m_1-\widetilde m_2 - \cdots - \widetilde m_{\lfloor \frac{d+1}{2} \rfloor} + \widetilde m_{\lfloor \frac{d+1}{2} \rfloor+1} + \dots + \widetilde m_d,
$$
and when $d$ is odd, we obtain
$$
-\widetilde m_1-\widetilde m_2 - \cdots - \widetilde m_{\lfloor \frac{d+1}{2} \rfloor-1} + \widetilde m_{\lfloor \frac{d+1}{2} \rfloor+1}+ \dots + \widetilde m_d.
$$ 
In either case, this simplifies as 
$$
C(\mathbf m)=\sum_{i=1}^{\lfloor d/2 \rfloor}
    \widetilde m_{d-i+1}-\widetilde m_i.
$$
\end{proof}

\begin{remark} 
A recent paper \cite{kline} defines a different cost function than ours for the earth mover's problem, namely $C'(\mathbf m):=\max\{m_i\}-\min\{m_i\}$.  We can see from Proposition \ref{costalt} that $C'$ agrees with our $C$ for $d=2$ and $d=3$, but not for $d>3$.  For example, letting $\mathbf m=(1,1,2,2)$, we have $C(\mathbf m)=2$ but $C'(\mathbf m)=1$.  For our purposes, we have chosen our $C$ because it counts \emph{every} earth-movement required to equalize the distributions.  For example, keeping $\mathbf m=(1,1,2,2)$, consider the distributions $\mu_1=\mu_2=(1,0)$ and $\mu_3=\mu_4=(0,1)$.  Then one solution is given by the array whose only nonzero entry is a $1$ at position $\mathbf m$.  Intuitively, we want the EMD of these four distributions to be $2$, not $1$, since we must first move a unit of earth by $1$ bin, and then move $another$ unit by 1 bin.
\end{remark}

Having built up the necessary intuition and formulas, we are finally ready to make our main definition:

\begin{definition}
Let $\bmu$ be a $d$-tuple of probability distributions, as above. Then the \textbf{generalized earth mover's distance} is defined as
\begin{equation}
\label{emddef}
    \EMD_d(\bmu):=\min_{J\in \mathcal J_{\bmu}}\sum_{\mathbf m \in [n]^d}
    C(\mathbf m)J(\mathbf m).
\end{equation} 
\end{definition}

\subsection{Existence of a greedy algorithm}

As mentioned in the first two sections, finding the right-hand side of \eqref{emddef} is equivalent to finding the optimal solution to a $d$-dimensional transport problem.  It is shown in \cite{bein} that there exists a greedy algorithm to find this solution in $O(d^2n)$ time, precisely when the cost array $C$ has the \textbf{Monge property} mentioned in the introduction:

\begin{definition}
A $d$-dimensional array $A$ has the \textbf{Monge property} if, for all $\mathbf x = (x_1,\dots,x_d)$ and $\mathbf y = (y_1,\dots,y_d)$, we have
$$
A(\min\{x_1,y_1\},\dots,\min\{x_d,y_d\})+A(\max\{x_1,y_1\},\dots,\max\{x_d,y_d\})
\leq A(\mathbf x)+A(\mathbf y).$$
\end{definition}

We now state the crucial proposition, whose proof we will give in Section \ref{s:mongeproof}.
\begin{prop}
\label{monge}
The cost array $C$ defined in \eqref{cost} has the Monge property.
\end{prop}

 This proposition, then, guarantees the existence of a greedy algorithm to compute $\EMD_d$.  (This justifies our writing ``$\min$" instead of ``$\inf$" in our definition; we also could have used a compactness argument as in \cite{bw}.)  The greedy algorithm described in \cite{bein} is a generalization of the two-dimensional ``northwest corner rule." Just as in the $d=2$ case (see Section \ref{s:EMD2}), for generic $d$ this algorithm arrives at its solution in the form of an array $J \in \mathcal J_{\bmu}$ whose support is a \textbf{chain}, i.e., pairwise comparable under the product order on $[n]^d$. (In \cite{bw}, Proposition 4, the ``straightening" procedure that converts the support of any $J$ into a chain, without increasing the total cost, is valid precisely because the cost array $C(i,j)=|i-j|$ has the Monge property.)   Rather than describe this greedy algorithm, which is already well-known (see \cite{bein} or \cite{kline}), our goal is instead to find the expected value of $\EMD_d$.  To this end, the importance of the algorithm is the following:

\begin{corollary}
\label{chain}
The minimum in \eqref{emddef} occurs for some $J \in \mathcal J_{\bmu}$ whose support is a chain in $[n]^d$.\end{corollary}

Since there is nothing special about the condition $|\mu_i|=1$ from the perspective of transport problems, Corollary \ref{chain} also holds in a discrete setting using integer compositions in place of probability distributions.  We will take this discrete approach in the next section, where we use a highly efficient combinatorial method to find the optimal array $J$ for any $\bmu$.

\section{A discrete approach}
\label{s:discreet}

We follow the method from \cite{bw}, with a view toward constructing a generating function in the next section.  In place of $\mathcal P_n$, we temporarily turn our attention to $\mathcal C(s,n)$, the set of (weak) compositions of some positive integer $s$ into $n$ parts.  That is, elements of $\mathcal C(s,n)$ are $n$-tuples of nonnegative integers whose sum is $s$ (whereas before, the elements of $\mathcal P_n$ were $n$-tuples of nonnegative real numbers whose sum was $1$); we can also think of compositions as histograms. The cost function $C$, however, remains the same as before, since it still describes distances among the $n$ bins in each of the $d$ compositions.

In this section, $\bmu=(\mu_1,\dots,\mu_d)$ denotes a sequence of compositions $\mu_i \in \mathcal C(s,n)$.  It is tempting simply to adjust the definition of $\mathcal J_{\bmu}$ so that arrays in the set must have nonnegative integer entries summing to $s$; then we could just re-use the definition \eqref{emddef} to obtain a definition for the discrete $\EMD_d$.  Although this is one viable approach, nevertheless, in light of Corollary \ref{chain}, we need only consider those arrays whose support is a chain; therefore we will work with the following set of arrays from this point forward:
\begin{equation*}
\label{chainarrays}
    \mathcal J^s_{(n^d)}:=\left\{ J \in (\mathbb Z_{\geq 0})^{n\times \cdots \times n} \: \Bigg | \: \sum_\mathbf m J(\mathbf m)=s,\text{ and the support of }J\text{ is a chain in }[n]^d\right\}.
\end{equation*}

We will now show that each $d$-tuple of compositions $\bmu \in \mathcal C(s,n)\times \cdots \times \mathcal C(s,n)$ corresponds to a unique array $J_{\bmu} \in \mathcal J_{(n^d)}^s$; therefore, by the end of the next subsection, we will have a direct computational definition for the discrete version of the EMD, which avoids taking the mimimum over a set of arrays as we must in the definition \eqref{emddef} of the continuous EMD.  Once we have this  definition for the discrete EMD, we will be able to recover the continuous version by scaling all the $\mu_i$ by $1/s$.  Near the end of the paper, we will do exactly this, and then let $s \rightarrow \infty$, in order to translate discrete results back into the continuous setting.

\subsection{Generalized RSK correspondence}

The authors of \cite{bw} use the Robinson-Schensted-Knuth correspondence to great effect in order to determine a unique optimal matrix $J_{(\mu_1,\mu_2)}$ for an ordered pair of compositions $(\mu_1,\mu_2)$. We now apply this same idea to $d$ compositions in order to uniquely determine an optimal $d$-dimensional array.  This will allow us to calculate the discrete $\EMD_d$ directly (and even more efficiently, in many cases, than by using the greedy algorithm mentioned above).

For non-experts, we summarize here a special case of the correspondence.  (For full details, see Chapter 4 of \cite{fulton}.)  The Robinson-Schensted-Knuth (RSK) correspondence furnishes a bijection:
\begin{equation*} 
 \left\{\parbox{4.6cm}{ordered pairs of semistandard Young tableaux of the same shape, with entries in $[n]$}\right\} \longleftrightarrow \left\{ \parbox{4.2cm}{$n \times n$  matrices with \\nonnegative integer entries}\right\}_{\textstyle.}
\end{equation*}
For our purposes, we will restrict our attention to the special case of one-row tableaux, since any composition in $\mathcal C(s,n)$ corresponds uniquely to a one-row tableau containing $s$ boxes with entries from $[n]$.  As an example, consider the two compositions
\begin{align*}
    \mu_1&=(1,2,3,4),\\
    \mu_2 &= (5,0,2,3)
\end{align*}
in $\mathcal C(10,4)$.  We associate to each composition $\mu_i$ a one-row tableau $T(\mu_i)$, which we fill so that the entry $k$ appears $\mu_i(k)$ times:
\begin{align*}
\Yvcentermath1
    T(\mu_1)&=\young(1223334444) \\
    T(\mu_2)&= \young(1111133444)
\end{align*}
Considering these tableaux as the two rows of a $2\times s$ matrix, we then have
\begin{equation*}
    M_{(\mu_1,\mu_2)}=
  \begin{bmatrix}
    1&2&2&3&3&3&4&4&4&4\\
    1&1&1&1&1&3&3&4&4&4
    \end{bmatrix}_{\textstyle .}
\end{equation*}

Finally, we fill in an $n\times n$ array $J_{(\mu_1,\mu_2)}$ whose $(i,j)$ entry equals the number of times the column $\left[\begin{smallmatrix} i\\j\end{smallmatrix}\right]$ appears in $M_{(\mu_1,\mu_2)}$.  For example, $\left[\begin{smallmatrix} 4\\4\end{smallmatrix}\right]$ appears three times, so we write a $3$ in position $(4,4)$.  Filling in the rest of the array, we obtain the correspondence
\begin{equation*}
    (\mu_1,\mu_2) \longleftrightarrow \big(T(\mu_1),T(\mu_2)\big)\longleftrightarrow M_{(\mu_1,\mu_2)}\longleftrightarrow J_{(\mu_1,\mu_2)}=
    \begin{bmatrix}
    1&0&0&0\\
    2&0&0&0\\
    2&0&1&0\\
    0&0&1&3
    \end{bmatrix}_{\textstyle.}
\end{equation*}
Note that we can also reverse the procedure, starting with the array $J_{(\mu_1,\mu_2)}$, translating its entries into a two-row matrix, and finally recovering the original pair of tableaux (and hence the pair of compositions).  Therefore this is indeed a bijection.  In the context of the EMD, the array $J_{(\mu_1,\mu_2)}$ has two significant properties:
\begin{itemize}
    \item The row and column sums coincide with the original compositions $\mu_1$ and $\mu_2$, so $J_{(\mu_1,\mu_2)}$ is a solution to the discrete earth mover's problem for $\mu_1$ and $\mu_2$.
    \item Since both rows of $M_{(\mu_1,\mu_2)}$ are nondecreasing, the support of $J_{(\mu_1,\mu_2)}$ is a chain in $[n]\times[n]$.
\end{itemize}
In summary, we have the following bijective correspondence in the case $d=2$:
\begin{equation*}
    \mathcal C(s,n)\times \mathcal C(s,n) \longleftrightarrow \mathcal J_{(n^2)}^s
\end{equation*}

This RSK correspondence extends naturally to $d$-tuples of compositions in $\mathcal C(s,n)$.  (For experts, details about the existence of this multivariate RSK generalization can be found in \cite{caselli}.)
Given $\bmu=(\mu_1,\dots,\mu_d)$, the tableaux corresponding to the $\mu_i$ uniquely determine a $d \times s$ matrix $M_{\bmu}$, which in turn determines a unique $n \times \cdots \times n$ array $J_{\bmu}\in\mathcal J_{(n^d)}^s$.  This correspondence is again bijective, establishing the following special case of the generalized RSK correspondence:
\begin{align}
\label{genrsk}
    \mathcal C(s,n)\times \cdots \times  \mathcal C(s,n) &\longleftrightarrow \mathcal J_{(n^d)}^s\\
    \bmu &\longmapsto J_{\bmu} \nonumber
\end{align}
This correspondence leads us to the following definition of the discrete EMD:
\begin{definition}
For positive integers $d$, $n$, and $s$, let $\bmu=(\mu_1,\dots,\mu_d)$ with each $\mu_i \in \mathcal C(s,n)$.  Let $J_{\bmu}$ the unique array corresponding to $\bmu$, as in \eqref{genrsk}.  Let $C$ be the cost function on $[n]^d$ as in \eqref{cost}.  Then we define the \textbf{discrete generalized earth mover's distance} to be
\begin{equation}
\label{defemddisc}
    \EMD_d^s(\bmu):=\sum_{\mathbf m \in [n]^d} C(\mathbf m) J_{\bmu}(\mathbf m),
\end{equation}
where we write the superscript $s$ to distinguish this discrete version from the continuous version.
\end{definition}

This definition in \eqref{defemddisc} is largely conceptual; in practice, we can calculate $\EMD_d^s(\bmu)$ directly from the matrix $M_{\bmu}$, since the support of $J_{\bmu}$ is determined by the columns of $M_{\bmu}$:

\begin{theorem}
\label{rskemd}
Let $\bmu=(\mu_1,\dots,\mu_d)$ with each $\mu_i \in \mathcal C(s,n)$, and let $C$ be the cost function in \eqref{cost}.  Let $M_{\bmu}$ be the unique $d \times s$ array corresponding to $\bmu$ via the generalized RSK correspondence, as described above, and let $M_{\bmu}(\bullet,j)$ denote the $j^\text{th}$ column vector in $M_{\bmu}$.  Then
\begin{equation*}
    \normalfont{\EMD}_d^s(\bmu)=\sum_{j=1}^s C(M_{\bmu}(\bullet,j)).
\end{equation*}
\end{theorem}
\begin{proof}
Consider the definition of $\EMD_d^s$ in \eqref{defemddisc}. By definition, $J_{\bmu}(\mathbf m)$ equals the number of occurrences of $\mathbf m$ as a column vector of the matrix $M_{\bmu}$. Therefore $J_{\bmu}(\mathbf m)=0$ unless $\mathbf m$ is one of those column vectors, and so we can simply sum over the $s$ column vectors to obtain the result.
\end{proof}

\begin{remark}
This construction via RSK is more efficient than the aforementioned greedy algorithm for computing $\EMD_d^s$ when $d$ or $n$ is sufficiently large: rather than filling a $d$-dimensional array in $O(d^2n)$ time, we need consider only the $s$ column vectors of a $d\times s$ matrix.
\end{remark}
\begin{example} 
As an example for $d=3$, consider the three compositions $$\mu_1=(4,0,1), \hspace{0.5cm} \mu_2=(1,2,2),\hspace{0.5cm} \mu_3=(0,5,0)$$ in $\mathcal C(5,3)$.  These correspond to the tableaux \young(11113), \young(12233), and \young(22222) respectively. Stacking these tableaux vertically gives us the matrix
$$M_{\bmu}=\begin{bmatrix}
1&1&1&1&3\\
1&2&2&3&3\\
2&2&2&2&2
\end{bmatrix}_{\textstyle.}$$ Now using Theorem \ref{rskemd} on the five columns of $M_{\bmu}$, we compute that
\begin{align*}
    \EMD_3^5(\bmu)&= C(1,1,2)+C(1,2,2)+C(1,2,2)+ C(1,3,2) + C(3,3,2)\\
    &=1+1+1+2+1\\
    &=\mathbf{6}.
\end{align*}
\end{example}

Before advancing to the main problem of the paper, we show that $\EMD_3$ can actually be expressed in terms of the classical $\EMD_2$.  (In the following proposition, we suppress the superscript $s$ because the result holds for both the discrete and the continuous version of EMD: the equality is independent of $s$, and therefore still holds after dividing both sides by $s$ and letting $s\rightarrow \infty$.)

\begin{prop}
\label{halfsum}
The value of $\normalfont \EMD_3$ is half the sum of the three pairwise $\normalfont\EMD_2$ values:
\begin{equation*}
\normalfont
    \EMD_3(\mu_1,\mu_2,\mu_3)=\frac{1}{2}\Bigg(\EMD_2(\mu_1,\mu_2)+\EMD_2(\mu_1,\mu_3)+\EMD_2(\mu_2,\mu_3)\Bigg)
\end{equation*}
\end{prop}

\begin{proof}
Let $\bmu=(\mu_1,\mu_2,\mu_3)$ as usual.  In each column $j$ of the matrix $M_{\bmu}$, call the three entries $a_j,b_j,c_j$, labeled so that $a_j\leq b_j \leq c_j$.  Each of the three pairs $(a_j,b_j)$, $(a_j,c_j)$, and $(b_j,c_j)$ corresponds naturally to one of the pairs $(\mu_1,\mu_2)$, $(\mu_1,\mu_3)$, and $(\mu_2,\mu_3)$, where row $i$ corresponds to $\mu_i$.  Therefore by Theorem \ref{rskemd}, we have
\begin{align*}
    \EMD_2(\mu_1,\mu_2)+\EMD_2(\mu_1,\mu_3)+\EMD_2(\mu_2,\mu_3)&=\sum_{j=1}^s C(a_j,b_j)+C(a_j,c_j)+C(b_j,c_j)\\
    &= \sum_j (b_j-a_j)+(c_j-a_j)+(c_j-b_j)\\
    &=\sum_j 2c_j - 2a_j\\
    &= 2 \sum_j c_j - a_j\\
    &= 2 \sum_j C(a_j,b_j,c_j)\\
    &= 2 \cdot \EMD_3(\mu_1,\mu_2,\mu_3).
\end{align*}
\end{proof}

This relationship does not generalize to $d>3$, because in general, the telescoping summand in the proof does not reduce in terms of a higher-dimensional cost function.  For example, when $d=4$, the analog of the third line above is $\sum_j 3d_j+c_j-b_j-3a_j$, or $\sum_j C(a_j,b_j,c_j,d_j)+2(d_j-a_j)$.

\section{Expected value of $\EMD_d$}
\label{s:expval}

Again we will follow and extend the methods used in \cite{bw} to arbitrary values of $d$.  First we will define a generating function of two variables to record the values of $\EMD_d^s$, which we will then differentiate in order to sum up all of these values.  This will allow us to compute expected value for $\EMD_d^s$ simply by reading off coefficients from a generating function of a single variable.

Because we are about to make a recursive definition, we will now need to consider $d$-tuples $\bmu$ consisting of compositions with different numbers of bins --- i.e., different values $n_i$ such that each $\mu_i\in \mathcal C(s,n_i)$. Therefore, $\mathbf n$ will denote this vector of bin numbers $(n_1,\dots,n_d)$, and we will write $(n^d)$ for the special vector $(n,\dots,n)$, which arises most frequently in applications.  

In order to encode an inclusion-exclusion argument, we will also need to define an indicator vector $\mathbf e(A)$ for a subset $A \subseteq [d]$.  Namely, \begin{equation*}
    \mathbf e(A):=\sum_{i\in A}\mathbf e_i
\end{equation*}
is the vector whose $i^\text{th}$ component is 1 if $i\in A$ and $0$ otherwise.  For example, if $d=5$, and $A=\{2,4,5\}$, then $\mathbf e(A)=(0,1,0,1,1)$.

\subsection{Generating function for the discrete case}

For fixed $\mathbf n=(n_1,\dots,n_d)$, we first define a generating function in two indeterminates $z$ and $t$:
\begin{equation}
    \label{hdef}
    H_\mathbf n(z,t):=\sum_{s=0}^\infty \left(\sum_{\bmu\in \mathcal C(s,n_1)\times \cdots \times \mathcal C(s,n_d)} z^{\EMD_d^s(\bmu)}\right)t^s.
\end{equation} We observe that the coefficient of $z^rt^s$ is the number of elements $\bmu\in \mathcal C(s,n_1)\times \cdots \times \mathcal C(s,n_d)$ such that $\EMD_d^s(\bmu)=r$.

A recursive definition of this generating function, for the $d=2$ case, is derived in \cite{bw}, Theorem 3.  Our generalization for $d>2$ follows: 

\begin{prop}
\label{hrecur}
Fix $\mathbf n=(n_1,\dots,n_d).$  The generating function $H_\mathbf n:=H_\mathbf n(z,t)$ has the following recursive definition, where the sum is over all nonempty subsets $A\subseteq [d]$:
$$
H_\mathbf n=
\frac{
\sum_A (-1)^{|A|-1}\cdot H_{\mathbf n - \mathbf e(A)}
}
{1-z^{C(\mathbf n)}t},
$$
where $H_{(1^d)}=\frac{1}{1-t}$, and $H_{\mathbf n-\mathbf e(A)}=0$ if $\mathbf n-\mathbf e(A)$ contains a $0$.
\end{prop}

\begin{proof}
Each $\bmu$ corresponds to a unique monomial
\begin{equation}
\label{monomial}
    \bmu \longleftrightarrow \prod_\mathbf m w_{\mathbf m}^{J_{\bmu}(\mathbf m)},
\end{equation}
where $\bmu \leftrightarrow J_{\bmu}$ is the RSK correspondence in \eqref{genrsk}.  The variables $w_\mathbf m$ are indexed by multi-indices $\mathbf m \in [n_1]\times \cdots \times [n_d]$.  Note that the degree of this monomial equals $s$ (the sum of the entries of $J_{\bmu}$).  Now making the substitution
\begin{equation}
\label{subst}
    w_\mathbf m \longmapsto z^{C(\mathbf m)}t,
\end{equation}
the above correspondence gives us the map
\begin{equation*}
    \bmu \longleftrightarrow \prod_\mathbf m w_{\mathbf m}^{J_{\bmu}(\mathbf m)} \longmapsto z^{\EMD_d^s(\bmu)}t^s.
\end{equation*}
Therefore the generating function $H_\mathbf n$ is just the image of the formal sum $H^*_\mathbf n$ of all monomials of the form \eqref{monomial}, under the substitution \eqref{subst}; as $s$ ranges over all nonnegative integers, there is one monomial in $H^*_\mathbf n$ for each possible $\bmu \in \mathcal C(s,n_1)\times \cdots \times \mathcal C(s,n_d)$. 

Since $\mathbf m \preceq \mathbf n$ for all $\mathbf m$, every array $J_{\bmu}$ is allowed to contain $\mathbf n$ in its support without violating the chain condition.  This means that every monomial in $H^*_\mathbf n$ is allowed to contain the variable $w_\mathbf n$, and so we may factor out the sum of all possible powers of $w_\mathbf n$, rewriting as
\begin{equation*}
    H^*_\mathbf n= \sum_{\bmu} \left(\prod_\mathbf m w_\mathbf m^{J_{\bmu}(\mathbf m)}\right) = \sum_{r=0}^\infty w^r_\mathbf n \cdot f(\mathbf w_{\mathbf m \neq \mathbf n}) = \frac{f(\mathbf w_{\mathbf m \neq \mathbf n})}{1-w_\mathbf n},
\end{equation*}
where $f$ is an infinite formal sum of monomials in the variables $w_\mathbf m$ where $\mathbf m \neq \mathbf n$.  Now we focus on rewriting this numerator $f$.  Suppose we subtract 1 from exactly one of the coordinates of $\mathbf n$; the possible results are $\mathbf n-\mathbf e(i)$ for $i=1,\dots,d$.  Now, on one hand, any monomial in $f$ containing the variable $w_{\mathbf n-\mathbf e(i)}$ appears in $H^*_{\mathbf n-\mathbf e(i)}$.  But on the other hand, note that all of these $\mathbf n-\mathbf e(i)$ are mutually incomparable under the product order, and so at most \emph{one} of them can be in the support of some $J_{\bmu}$, because of the chain condition.  Therefore any monomial in $f$ contains at most \emph{one} of the variables $w_{\mathbf n-\mathbf e(i)}$.  But the sum $\sum_{i=1}^d H^*_{\mathbf n-\mathbf e(i)}$ still overcounts the monomials appearing in $f$, since the same monomial may appear in several distinct summands.  

In other words, we want $f$ to be the formal sum of the \emph{union} (without multiplicity) of the monomials which appear in the summands $H^*_{\mathbf n-\mathbf e(i)}$.  We can achieve this by using the inclusion-exclusion principle: subtract those monomials which appear in at least 2 of the summands, then add back the monomials which appear in at least 3 of the summands, then subtract those appearing in at least 4 summands, and so on, until we arrive at those monomials appearing in all $d$ of the summands.  We can write this inclusion-exclusion as an alternating sum over nonempty subsets $A\subseteq [d]$, adding when $|A|$ is odd and sutracting when $|A|$ is even:
\begin{equation*}
    f=\sum_A (-1)^{|A|-1}\cdot H^*_{\mathbf n - \mathbf e(A).}
\end{equation*}
Finally, applying the substitution \eqref{subst}, we obtain
\begin{align*}
    H_\mathbf n = H^*_\mathbf n \Big|_{w_\mathbf m = z^{C(\mathbf m)}t}&= \frac{\sum_A (-1)^{|A|-1}\cdot H^*_{\mathbf n - \mathbf e(A)}}{1-w_\mathbf n} \Bigg |_{w_\mathbf m = z^{C(\mathbf m)}t}\\
    &=\frac{\sum_A(-1)^{|A|-1}\cdot H_{\mathbf n - \mathbf e(A)}}{1-z^{C(\mathbf n)}t,}
\end{align*}
proving the recursion.

As for the base case $H_{(1,\dots,1)}=\frac{1}{1-t}$, there is only one element in $\mathcal C(s,1)$, and so since every $n_i=1$, the inside sum in \eqref{hdef} has only one term; moreover, this unique $\bmu$ is just $d$ copies of the same trivial composition of $s$ into $1$ part, meaning that $\EMD_d^s(\bmu)=0$.  Hence $H_{(1,\dots,1)}(z,t)=\sum_s z^0t^s=\sum_s t^s$, whose closed form is $\frac{1}{1-t}$.  Likewise, since $\mathcal C(s,0)$ is empty, we must have $H=0$ if any of the $n_i$ become $0$.
\end{proof}

\begin{example} We write out this recursive definition in a concrete case, where $d=3$ and $\mathbf n = (5,2,2)$.  It is easiest to order the terms of the numerator according to the size of the subset $A$.  First, for $|A|=1$, we add together all possible $H_{\mathbf n'}$, where $\mathbf n'$ equals $\mathbf n$ with exactly 1 coordinate decreased; then for $|A|=2$, we \emph{subtract} all possible $H_{\mathbf n''}$ where $\mathbf n''$ equals $\mathbf n$ with exactly 2 coordinates decreased; finally, for $|A|=3$, we \emph{add} the one possible $H_{\mathbf n'''}$ where $\mathbf n'''$ equals $\mathbf n$ with all 3 coordinates decreased.  As for the denominator, $C$ is the same cost function we defined in \eqref{cost}, meaning that $C(5,2,2)=5-2=3$.   Then the recursion for $H_\mathbf n$ looks like this:
\begin{equation*}
    H_{(5,2,2)}=\frac{H_{(4,2,2)}+H_{(5,1,2)}+H_{(5,2,1)}-H_{(4,1,2)}-H_{(4,2,1)}-H_{(5,1,1)}+H_{(4,1,1)}}{1-z^3 t}
\end{equation*}
\end{example}

Having seen an example, we now mention the important (and very well-studied) specialization that results from setting $z=1$.  In this case, the coefficient of $t^s$ in $H_\mathbf n(1,t)$ is simply the total number of $d$-tuples $\bmu$, which is $\prod_{i=1}^d|\mathcal C(s,n_i)|=\prod_{i=1}^d \binom{s+n_i-1}{n_i-1}$:
\begin{equation}
\label{hn1t}
    H_\mathbf n(1,t)=\sum_{s=0}^\infty \prod_{i=1}^d \binom{s+n_i-1}{n_i-1}t^s.
\end{equation}

It is shown in \cite{dillon} that the closed form of \eqref{hn1t} is, after adjusting the index to match our setup, and writing $|\mathbf n|:=n_1+\cdots +n_d$,
\begin{equation}
\label{hnclosed}
H_\mathbf n(1,t)=\frac{W(t)}{(1-t)^{|\mathbf n|-d+1}},
\end{equation}
where the numerator $W(t)$ is a polynomial whose coefficients are the ``Simon Newcomb" numbers.  (For more on this natural generalization of Eulerian numbers to multisets, see \cite{abramson}, \cite{dillon}, and \cite{morales}.) Specifically, denoting the coefficient of $t^i$ in $W_\mathbf n$ by the symbol $[t^i]W_\mathbf n$, and adopting the $A$-notation originally used in \cite{dillon}, we have
\begin{align*}
[t^i]W_\mathbf n &= A\big(\mathbf n-(1,\dots,1),i\big)\\
&:= \text{ \# permutations of the multiset }\{1^{n_1-1},\dots,d^{n_d-1}\} \text{ containing }i\text{ descents}\\
&=\sum_{j=0}^i (-1)^j\binom{|\mathbf n|-d+1}{j}\prod_{k=1}^d\binom{i-j+n_k-1}{n_k-1}.
\end{align*}
The degree of the polynomial $W_\mathbf n$ is shown in \cite{dillon} to be $$\sum_{i=1}^d(n_i-1)-\max\{n_1,\dots,n_d\}.$$  The combinatorial interpretation implies that the coefficients of $W_\mathbf n$ are positive; in the special case where $\mathbf n = (n^d)$, then $W_\mathbf n$ is also unimodal and palindromic.  (This can be shown from a combinatorial or ring-theoretic approach; for the latter, see \cite{morales}, or Chapter 5 of \cite{bruns} on Stanley-Reisner and Gorenstein rings.)

From the combinatorial description above of $[t^i]W_\mathbf n$, it follows that the evaluation $W_\mathbf n(1)$ equals the total number of permutations of the multiset $\{1^{n_1-1},\dots,d^{n_d-1}\}$:
\begin{equation}
    \label{w1}
     W_\mathbf n(1)=\frac{\left(\sum_{i=1}^d (n_i-1)\right)!}{\prod_{i=1}^d (n_i-1)!} =\frac{(|\mathbf n |-d)!}{\prod (n_i-1)!}
\end{equation}
(We will need this fact later.)  More geometrically, every permutation of the multiset $\{1^{n_1-1},\dots,d^{n_d-1}\}$ corresponds to a unique increasing lattice path in $\mathbb N^d$, beginning at $(1,\dots,1)$ and ending at $(n_1,\dots,n_d)$: reading left to right, each occurrence of $i$ in the permutation signifies adding the standard basis vector $\mathbf e_i$ to the current position in the path.  Therefore, $W(1)$ can be interpreted as the total number of increasing paths connecting opposite corners of an $n_1 \times \cdots \times n_d$ array --- in other words, the number of chains in $[n_1] \times \cdots \times [n_d]$ with maximal length.


\subsection{A partial derivative}

Next, in order to transfer the $\EMD$ values from the exponents of $z$ into coefficients, we take the partial derivative of $H_\mathbf n$ with respect to $z$.  Applying the quotient rule to our definition of $H_\mathbf n$ in Proposition \ref{hrecur}, we obtain the following, where the sum still ranges over nonempty subsets $A\subseteq [d]$:
\begin{equation*}
    \frac{\partial H_\mathbf n}{\partial z}=\frac{\displaystyle \left(1-z^{C(\mathbf n)} t\right) \left( \sum_A (-1)^{|A|-1}\cdot \frac{\partial H_{\mathbf n-\mathbf e(A)}}{\partial z}\right) + C(\mathbf n) \cdot z^{C(\mathbf n)-1}\cdot t \cdot \left(\sum_A (-1)^{|A|-1}\cdot H_{\mathbf n-\mathbf e(A)}\right)}{\left( 1 -z^{C(\mathbf n)}t\right)^2}
\end{equation*}
\normalsize Now that the exponents have been changed into coefficients of $z$, we can set $z=1$:
\begin{align}
    \label{h'}
H'_\mathbf n:=\frac{\partial H_\mathbf n}{\partial z}\Bigg |_{z=1} &= \sum_{s=0}^\infty \left( \sum_{\bmu\in \mathcal C(s,n_1)\times \cdots \times \mathcal C(s,n_d)}\EMD_d^s(\bmu)\right) t^s\nonumber\\
&=
\frac{\displaystyle (1-t) \left( \sum_A (-1)^{|A|-1}\cdot H'_{\mathbf n-\mathbf e(A)}\right) + t \cdot C(\mathbf n) \left(\sum_A (-1)^{|A|-1}\cdot H_{\mathbf n-\mathbf e(A)}\right)}{( 1 -t)^2}
\end{align}
At this point, $z$ has played out its role, and so from now on we will write $H_\mathbf n$ in place of $H_\mathbf n(1,t)$.

Note that the coefficient of $t^s$ in $H'_\mathbf n$ is the sum of the values $\EMD_d^s(\bmu)$ for all valid $d$-tuples $\bmu$.  This means that our goal is now in sight: to find the expected value of $\EMD_d^s$ for fixed $\mathbf n$, we need to divide the sum of all possible $\EMD_d^s$ values (i.e., the coefficient of $t^s$ in $H'_\mathbf n$) by the total number of possible inputs $\bmu$ (i.e., the coefficient of $t^s$ in $H_\mathbf n$).  Therefore, once we find a way to simplify \eqref{h'}, we will be able to compute the result
\begin{equation}
\label{expected}
    \mathbb E(\EMD_d^s)=\frac{[t^s]H'_\mathbf n}{[t^s]H_\mathbf n} = \frac{[t^s]H'_\mathbf n}{\prod_{i=1}^d \binom{s+n_i-1}{n_i-1}},
\end{equation}
where $[t^s]$ denotes the coefficient of $t^s$ in a series.

In order to make the expression \eqref{h'} for $H'_\mathbf n$ more tractable to program, we will now focus only on the numerators of $H_\mathbf n$ and $H'_\mathbf n$.    We have already determined $W_\mathbf n(t)$, the numerator for $H_\mathbf n$, in the previous subsection.  We will let $N(t)$ denote the numerator of $H_\mathbf n'$.  By using software and observing patterns for small $\mathbf n$, we anticipate that the denominator of $H_\mathbf n'$ has exponent $|\mathbf n|-d+2$, and so we now set both
\begin{equation}
\label{hn}
W_\mathbf n:=(1-t)^{|\mathbf n |-d+1}H_\mathbf n \hspace{1cm} \text{and} \hspace{1cm} N_\mathbf n(t):= (1-t)^{|\mathbf n|-d+2}H'_\mathbf n.
\end{equation}
Therefore, we can clear denominators in \eqref{h'} by multiplying both sides by $(1-t)^{|\mathbf n|-d+2}$.  Proceeding carefully and clearing the remaining denominators using \eqref{hn}, the pattern becomes clear:
\begin{align}
    N_\mathbf n&=\sum_A (-1)^{|A|-1} (1-t)^{|A|-1} N_{\mathbf n-\mathbf e(A)}+t \cdot C(\mathbf n) \cdot (1-t)^{|\mathbf n|-d}\cdot (1-t)\cdot H_\mathbf n \nonumber\\
    &=\sum_A (t-1)^{|A|-1} N_{\mathbf n-\mathbf e(A)}+t \cdot C(\mathbf n) \cdot W_\mathbf n \label{n}
\end{align}
This provides us with a quick recursive code to obtain $N_\mathbf n$, after which we need only divide by $(1-t)^{|\mathbf n|-d+2}$ to recover $H'_\mathbf n$.  The rest is just a matter of extracting coefficients in order to apply the result \eqref{expected}.


\subsection{Expected value for continuous version of $\EMD_d$}

Now that we have a way to determine the expected value for the discrete EMD, we aim to find a formula for the expected value in the continuous setting.  

Starting with the expected value from \eqref{expected}, we scale by $1/s$ to normalize, and then let $s$ grow asymptotically:
\begin{align*}
\mathcal E_\mathbf n :=\mathbb E(\EMD_d) &= \lim_{s\rightarrow \infty} \frac{1}{s} \cdot \mathbb E(\EMD_d^s)\\[1ex]
&= \lim_{s\rightarrow \infty} \frac{1}{s} \cdot \frac{[t^s]H'_\mathbf n}{\prod_{i=1}^d \binom{s+n_i-1}{n_i-1}.}
\end{align*}
First we focus on the $[t^s]H'_\mathbf n$ part, namely the coefficient of $t^s$ in $H'_\mathbf n=\frac{N_\mathbf n(t)}{(1-t)^{|\mathbf n|-d+2}}$.  Now, the coefficient of $t^s$ in the series $\frac{1}{(1-t)^{|\mathbf n |-d+2}}$ is just \begin{equation*} 
\binom{s+|\mathbf n|-d+1}{|\mathbf n|-d+1}= \frac{s^{|\mathbf n|-d+1}}{(|\mathbf n |-d+1)!}+ \text{ lower-order terms in } s.
\end{equation*}  Meanwhile, $N_\mathbf n(t)$ is a polynomial, with some finite degree $b$.  Now, as $s\rightarrow \infty$, we have $s-b\rightarrow s$, and so the coefficient of $t^s$ in $H'_\mathbf n$ is asymptotic to $\frac{s^{|\mathbf n|-d+1}}{(|\mathbf n |-d+1)!}$ multiplied by the sum of the coefficients of $N_\mathbf n(t)$.  But this sum is just $N_\mathbf n(1)$, and so we have: $$
[t^s]H'_\mathbf n \sim N_\mathbf n(1) \cdot \frac{s^{|\mathbf n|-d+1}}{(|\mathbf n |-d+1)!}
$$
Accounting for the $1/s$, we currently have the following:
$$\mathcal E_\mathbf n=\lim_{s\rightarrow \infty}N_\mathbf n(1)\cdot \frac{s^{|\mathbf n|-d}}{(|\mathbf n|-d+1)! \prod_{i=1}^d \binom{s+n_i-1}{n_i-1}}$$
Now, since $\prod_i\binom{s+n_i-1}{n_i-1}\sim \prod_i\frac{s^{n_i-1}}{(n_i-1)!}=\frac{s^{|\mathbf n|-d}}{\prod_i(n_i-1)!}$, this becomes
\begin{equation}
\label{e_n penult}
\mathcal E_\mathbf n=N_\mathbf n(1)\cdot \frac{\prod_{i=1}^d (n_i-1)!}{(|\mathbf n|-d+1)!}
\end{equation}
But when we evaluate $N_\mathbf n(1)$ from equation \eqref{n}, the terms with $(t-1)$ all disappear; hence we need only consider subsets $A\subseteq [d]$ with one element, meaning we are now summing from $1$ to $d$:
$$
N_\mathbf n(1)=\sum_{i=1}^d N_{\mathbf n - \mathbf e(i)}(1)+C(\mathbf n)W_\mathbf n(1)
$$
Substituting for $W_\mathbf n(1)$ using \eqref{w1}, we have 
$$
N_\mathbf n(1)=\sum_{i=1}^d N_{\mathbf n - \mathbf e(i)}(1)+ \frac{C(\mathbf n)\cdot(|\mathbf n|-d)!}{\prod_{i=1}^d(n_i-1)!}
$$
Finally, returning to \eqref{e_n penult} and plugging this all in for $N_\mathbf n(1)$, we conclude with the recursive definition
\begin{align}
    \mathcal E_\mathbf n &=
    \left[\sum_{i=1}^d N_{\mathbf n - \mathbf e(i)}(1)+ \frac{C(\mathbf n)\cdot(|\mathbf n|-d)!}{\prod_{i=1}^d(n_i-1)!}\right]\cdot \frac{\prod_{i=1}^d (n_i-1)!}{(|\mathbf n|-d+1)!} \nonumber\\[2ex]
    &= \frac{
    \sum_{i=1}^d N_{\mathbf n - \mathbf e(i)}(1)\cdot 
    \frac{
    \prod_{i=1}^d(n_i-1)!}
    {(|\mathbf n|-d)!}}{|\mathbf n |-d+1}+\frac{C(\mathbf n)\cdot (|\mathbf n|-d)!}{(|\mathbf n|-d+1)!}\nonumber\\[2ex]
    &=\frac{\sum_{i=1}^d (n_i-1)\mathcal E_{\mathbf n-\mathbf e(i)}+C(\mathbf n)}{|\mathbf n |-d+1}, \label{expectfinal}
\end{align}
where $\mathcal E_{\mathbf n - \mathbf e(i)}=0$ if $\mathbf n-\mathbf e(i)$ contains a $0$. 

We record this as the main theorem of this paper, in the most useful case where $\mathbf n=(n^d)$:
\begin{theorem}
\label{expval}
The expected value of $\normalfont{\EMD}_d$ on $\mathcal P_n \times \cdots \times \mathcal P_n$ is $\mathcal E_{(n^d)}$ as defined in \eqref{expectfinal}.
\end{theorem}
\begin{remark}
Recall from Proposition \ref{halfsum} the special relationship between $\EMD_3$ and $\EMD_2$, namely, $\EMD_3$ equals half the sum of the three pairwise $\EMD_2$ values.  This leads us to anticipate that
\begin{align*}
    \mathcal E_{(n^3)}=\mathbb E(\EMD_3) &= \mathbb E\left(\frac{1}{2}(\EMD_2+\EMD_2+\EMD_2)\right)\\
    &=\mathbb E\left( \frac{3}{2}\EMD_2\right)\\
    &=\frac{3}{2}\mathbb E(\EMD_2)\\[1ex]
    &=\frac{3}{2}\mathcal E_{(n^2).}
\end{align*}
We confirm this in Mathematica:

\begin{center} 
\begin{tabular}{c||c|c|c}
  $n$&$\mathcal E_{(n^2)}$   &  $\mathcal E_{(n^3)}$ & $\mathcal E_{(n^3)}/\mathcal E_{(n^2)}$\\ \hline \hline
   2&0.3333&0.5000&1.5\\
   3&0.5333&0.8000&1.5\\
   4&0.6857&1.0286&1.5\\
   5&0.8127&1.2191&1.5\\
   6&0.9235&1.3853&1.5\\
   7&1.0230&1.5345&1.5\\
   8&1.1139&1.6709&1.5\\
   9&1.1982&1.7972&1.5\\
   10&1.2770&1.9155&1.5
\end{tabular}
\end{center}
\end{remark}

\subsection{Unit normalized EMD} It is often convenient to \textit{unit normalize} the $\EMD_d$ so that its value falls between 0 and 1.  To this end, we claim that for a given $n$, the maximum value of $\EMD_d$ is $\lfloor d/2 \rfloor (n-1)$.  To see this, observe that the maximum value of the discrete $\EMD^s_d(\bmu)$ occurs when in every column of the matrix $M_{\bmu}$ (given by the RSK correspondence), half the entries are $1$'s and the other half are $n$'s; if $d$ is odd, then ``half" means $\lfloor d/2 \rfloor$, with the leftover entry being irrelevant by Proposition \ref{costalt}.  For such a $\bmu$, then, $\EMD^s_d(\bmu)$ equals the cost $\lfloor d/2 \rfloor (n-1)$ multiplied by $s$ (the number of columns).  After dividing by $s$ to pass to the continuous setting, we see that the maximum value of $\EMD_d$ is $\lfloor d/2 \rfloor (n-1)$, as claimed.  Therefore we present definitions for the \textbf{unit normalized} $\EMD_d$ and its expected value:
\begin{equation}
\label{unitnorm}
    \widehat{\EMD}_d(\bmu):=\frac{\EMD_d(\bmu)}{\lfloor d/2 \rfloor (n-1)} \hspace{1cm} \text{and} \hspace{1cm}     \widehat{\mathcal E}_{(n^d)}:=\frac{\mathcal E_{(n^d)}}{\lfloor d/2 \rfloor (n-1).}
\end{equation}

We observe a curious phenomenon when we fix $n$ and let $d$ increase: the unit normalized expected value alternately increases ($d$ changing from even to odd) and decreases ($d$ changing from odd to even), as seen in this example for $n=3$:

\bigskip
\begin{center}
{\renewcommand{\arraystretch}{1.5} 
\begin{tabular}{|c|c|c|c|c|c|c|c|c|c|}\hline
$d$ & 2 & 3 & 4 & 5 & 6 & 7 & 8 & 9 & 10\\ \hline
$\widehat{\mathcal E}$ & 0.2667 & 0.4000 & 0.3175 & 0.3968 & 0.3388 & 0.3952 & 0.3505 & 0.3943 & 0.3579 \\ \hline
\end{tabular}}
\end{center}
\bigskip

In some sense, then, an even number of distributions are more likely to be ``closer" together than an odd number of distributions.  The histograms of $\EMD^s_d$ values confirm this impression (see Figure \ref{fig:histograms}): when $d$ is even, the histograms are clearly right-skewed, whereas when $d$ is odd, the histograms have nearly zero skew (although still slightly right-skewed).  It seems that this alternating phenomenon is a consequence of taxicab geometry in even vs.\ odd dimensions: specifically, when $d$ is odd, the median of a vector's coordinates contributes nothing to its distance from the main diagonal (i.e., the cost function $C$).  Intuitively, when the other coordinates have a wide range --- that is, when $C$ is relatively large --- this ``free" coordinate can assume more values without affecting $C$, than when the range of the coordinates is smaller.  This increases the proportion of high-cost array positions compared to the case when $d$ is even.

\begin{figure}
\captionsetup[subfigure]{labelformat=empty}
    \hspace{7pt}\begin{subfigure}[b]{0.45\textwidth}
    \includegraphics[width=\textwidth]{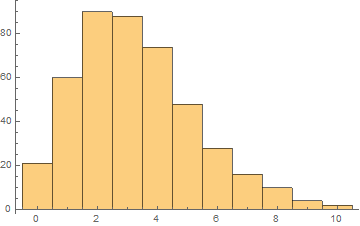}\caption{$d=2$}
    \end{subfigure}
    \begin{subfigure}[b]{0.45\textwidth}
    \includegraphics[width=\textwidth]{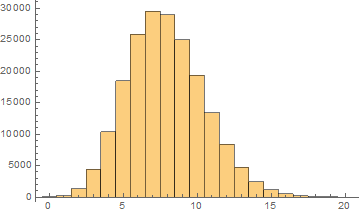}\caption{$d=4$}
    \end{subfigure}
    
    \begin{subfigure}[b]{0.45\textwidth}\includegraphics[width=\textwidth]{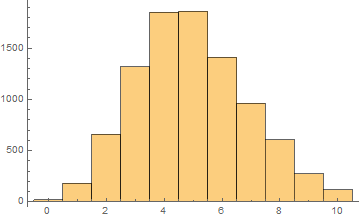}\caption{$d=3$}
    \end{subfigure}
    \begin{subfigure}[b]{0.45\textwidth}
    \includegraphics[width=\textwidth]{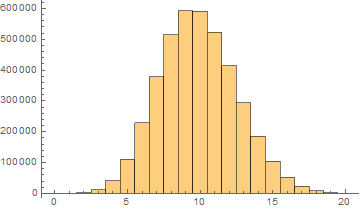}\caption{$d=5$}
    \end{subfigure}
    
    \caption{Histograms of discrete $\EMD_d^s$ values, fixing $s=5$ and $n=3$.  Note the more severe skew to the right when $d$ is even, compared to nearly zero skew when $d$ is odd.}
    \label{fig:histograms}
    
\end{figure}

\section{Real-world data} 
\label{s:realworld}

As a basic example, we now apply this generalized discrete EMD in order to compare the grade distributions in the spring vs.\ fall semester of 2019, for the course MATH 232 (Calculus II) at the University of Wisconsin-Milwaukee.  This data is contained in the Section Attrition and Grade Report, published by the Office of Assessment and Institutional Research at UWM; final letter grades are A, B, C, D, and F, so $n=5$.  We consider all the sections of MATH 232 with more than 20 students enrolled; there were seven such sections in each semester in 2019, and so we put $d=7$ in both cases.  We list the grade distributions below, with $\bmu$ corresponding to the spring sections and $\bm{\nu}$ the fall sections.  (We will continue to use the term ``distribution" rather than ``composition" in the context of this application.)  

The value of $s$ is problematic, since different classes in the real world do not have the same number of students.  One solution is to scale each distribution so that all distributions share a common $s$-value, namely, the least common multiple of their individual $s$-values.  On one hand, this method retains whole-number distributions whose grade proportions remain unchanged; but on the other hand, with seven classes of roughly 30 students each, that least common multiple could easily be in the billions, so this approach is computationally impractical.  Instead, we have chosen to scale each distribution so that the sum of its components equals the maximum of the seven original $s$-values; this of course produces non-integer distributions, and so in order to use Theorem \ref{rskemd}, we then round each component up or down, while respecting the original proportions as much as possible, until all distributions share the same $s$-value.  In our case, the common $s$-value for the $\mu_i$ is 31, and the common $s$-value for the $\nu_i$ is 33.

We list the resulting distributions, with the letter grades in descending order from A to F:
\bigskip
\begin{center}
\begin{tabular}{|ll|l|}\hline
   \textbf{Spring 2019}  && \textbf{Fall 2019} \\ \hline
   $\mu_1=(7,6,5,7,6)$  && $\nu_1=(2,5,11,9,6)$ \\
   $\mu_2=(12,6,8,3,2)$ && $\nu_2=(4,8,11,3,7)$ \\
   $\mu_3=(4,6,6,10,5)$ && $\nu_3=(4,7,12,7,3)$ \\
   $\mu_4=(6,5,10,6,4)$ && $\nu_4=(6,10,9,6,4)$ \\
   $\mu_5=(6,9,8,6,2)$ && $\nu_5=(4,16,9,1,3)$ \\
   $\mu_6=(6,7,8,5,5)$ && $\nu_6=(5,6,10,3,9)$ \\
   $\mu_7=(8,6,5,8,4)$ && $\nu_7=(5,8,9,7,4)$ \\\hline
   
\end{tabular}
\end{center}
\bigskip
Converting Theorem \ref{rskemd} into Mathematica code, we quickly compute (in 0.03 seconds each):
\begin{equation*}
    \EMD_7^{31}(\bmu)=49 \hspace{1cm} \text{and} \hspace{1cm} \EMD_7^{33}(\bm{\nu})=56.
\end{equation*}
Dividing by $s$ in each case to normalize (so that $\bmu$ and $\bm{\nu}$ actually mean $\frac{1}{s}\bmu$ and $\frac{1}{s}\bm{\nu}$ now),
\begin{equation*}
    \EMD_7(\bmu)=1.58065 \hspace{1cm} \text{and} \hspace{1cm} \EMD_7(\bm{\nu})=1.69697,
\end{equation*}
and finally dividing by $\lfloor d/2 \rfloor (n-1)=3 \cdot 4 =12$ in order to unit normalize, as in \eqref{unitnorm}, we obtain
\begin{equation*}
    \widehat{\EMD}_7(\bmu)=\bm{0.131721} \hspace{1cm} \text{and} \hspace{1cm} \widehat{\EMD}_7(\bm{\nu})=\bm{0.141414}.
\end{equation*}
To compare these results to the expected value, we code the recursive definition in Theorem \ref{expval}; then upon unit normalizing as in \eqref{unitnorm}, we find that
\begin{equation*}
    \widehat{\mathcal E}_{(5^7)}=\bm{0.298621}.
\end{equation*}

From this very limited data set, at least, it is clear that not only are the $\EMD$ values extremely consistent between the two semesters (within 0.01 of each other), but also they are significantly less than the expected value --- less than \textit{half} the expected value, in fact.  We should not be too surprised by this, of course, since college grades are (hopefully) not assigned at random.  It will be interesting to track the EMD for multiple courses in multiple semesters, with the goal of performing some informative cluster analysis on the results.

\section{Connections to algebraic geometry and representation theory}
\label{s:repthy}

In this section, we show that our generating function $H_\mathbf n$ from above is the Hilbert series of the Segre embedding from algebraic geometry.  In the $d=2$ case, this embedding is a determinantal variety, whose structure as an infinite-dimensional $\mathfrak{su}(p,q)$-module we can describe in the context of the EMD.

\subsection{A determinantal variety}

In this subsection we let $d=2$, and we will write $p,q$ in place of the usual $\mathbf n=(n_1,n_2)$.  In this case, as indicated in \cite{bw}, the series $H_{p,q}:=H_{p,q}(1,t)$, from \eqref{hn1t}, is in fact the Hilbert series of the \textbf{determinantal variety}
\begin{equation*}
    \mathcal D_{p,q}^{\leq 1}:=\left\{ M \in M_{p,q}(\mathbb C) \mid \text{rank }M \leq 1 \right\}
\end{equation*}
consisting of complex $p\times q$ matrices with rank at most 1.  To see this, it suffices to show that given a nonnegative integer $s$, the number of elements in $\mathcal C(s,p)\times \mathcal C(s,q)$ equals the dimension of $\mathbb C[\mathcal D_{p,q}^{\leq 1}]^s$, the space of homogeneous degree-$s$ polynomial functions on $\mathcal D_{p,q}^{\leq 1}$.  
To this end, let $w_{ij}$ be the coordinate functions on a generic $p\times q$ matrix. Since all $2 \times 2$ minors vanish for any element of $\mathcal D_{p,q}^{\leq 1}$, we observe that
\begin{equation*}
    \mathbb C[\mathcal D_{p,q}^{\leq 1}]\simeq \mathbb C[\mathbf w_{ij}] / \mathcal I,
\end{equation*}
where $\mathcal I$ is the \textbf{determinantal ideal} generated by the quadratics of the form $w_{ij}w_{i'j'}-w_{ij'}w_{i'j}$ for $i<i'$ and $j<j'$. It follows that a basis for $\mathbb C[\mathcal D_{p,q}^{\leq 1}]^s$ is given by the set of monomials
\begin{equation*}
    \mathcal B = \left\{ \prod_{k=1}^s w_{i_k,j_k} \:\Bigg | \: \{(i_k,j_k) \} \text{ is a chain in }[p]\times [q] \right\}.
\end{equation*}
Now we have an obvious bijective correspondence between $\mathcal B$ and the set $\mathcal J^s_{p,q}$ (borrowing notation from \eqref{chainarrays} to denote $p\times q$ arrays with support in a chain):
\begin{equation}
\label{monom_to_matrices}
    \prod_{i,j} w_{ij}^{J_{ij}} \longleftrightarrow J \in \mathcal J^s_{p,q}
\end{equation}
But $\mathcal J^s_{p,q}$ is in bijective correspondence with $\mathcal C(s,p)\times \mathcal C(s,q)$, as is clear from a slight generalization of our RSK correspondence in \eqref{genrsk}.  This proves our claim that $H_{p,q}$ is the Hilbert series of $\mathcal D_{p,q}^{\leq 1}$.

\subsection{The Segre embedding}

We extend the previous result to $d>2$, and so as before we fix $\mathbf n = (n_1,\dots,n_d)$.  The specialization $H_\mathbf n:=H_\mathbf n(1,t)$ of our generating function from \eqref{hn1t} happens also to be the Hilbert series of the Segre embedding:
\begin{align*}
\mathbb P(\mathbb C^{n_1})\times \cdots \times \mathbb P(\mathbb C^{n_d})&\hookrightarrow \mathbb P(\mathbb C^{n_1}\otimes \dots \otimes \mathbb C^{n_d}), \\
\left(\left[v^{(1)}\right],\dots,\left[v^{(d)}\right]\right)&\mapsto \left[v^{(1)}\otimes \cdots \otimes v^{(d)}\right].
\end{align*}
(See \cite{harris} and \cite{morales}.)  That is, $H_\mathbf n$ is the Hilbert series of the simple tensors. (In the case $d=2$, the set of simple tensors in $\mathbb C^p \otimes \mathbb C^q$ can be identified with the determinantal variety $\mathcal D_{p,q}^{\leq 1}$, coinciding with the previous subsection.)

To sketch this generalization of the $d=2$ case, which we presented in detail above, we let $\mathbf m$ range over all multi-indices $(m_1,\dots,m_d) \in  [n_1]\times \cdots [n_d]$, as in the proof of Proposition \ref{hrecur}.  Now consider a simple tensor $v^{(1)}\otimes \cdots \otimes v^{(d)}$.  We can expand this tensor in the standard basis as 
\begin{equation*} 
\sum_{\mathbf m}\underbrace{\left( v^{(1)}_{m_1} \cdots v^{(d)}_{m_d}\right)}_{w_\mathbf m} \mathbf e_{m_1}\otimes \cdots \otimes \mathbf e_{m_d},
\end{equation*} 
where $v^{(k)}_{\ell}$ is the $\ell^\text{th}$ coordinate of the vector $v^{(k)}$; the $w_\mathbf m$ are coordinate functions.  For any two multi-indices $\mathbf m$ and $\mathbf m'$, we see that the quadratic $w_\mathbf m w_{\mathbf m'}$ is invariant under the exchange of indices component-wise between $\mathbf m$ and $\mathbf m'$.  Intuitively, then, we can again mod out by the determinantal ideal generated by all $2\times 2$ minors, just as we did in the $d=2$ case above.

The upshot is that a basis for the coordinate ring of the simple tensors is given by those monomials $w_{\mathbf m_1}\cdots w_{\mathbf m_s}$ such that the set of multi-indices $\{\mathbf m_i\}$ form a chain.  We therefore have the generalization of \eqref{monom_to_matrices}, so we conclude that $H_\mathbf n$ is the Hilbert series of the simple tensors.  (Technically, of course, the Segre embedding is contained in the \emph{projectivization} of the simple tensors, so we are sweeping a fair amount under the rug here.)

\subsection{Representation theory}

Returning to the $d=2$ case, the coordinate ring $\mathbb C[\mathcal D_{p,q}^{\leq 1}]$ is an infinite-dimensional representation of the Lie algebra $\mathfrak{su}(p,q)$, known as the first Wallach representation.  (See \cite{ew}.)  We will show how the action on this representation corresponds to manipulating the two compositions in our $\EMD^s_2$ setting.

Consider the polynomial ring $\mathbb C[\mathbf x,\mathbf y]:=\mathbb C[x_1,\dots,x_p,y_1,\dots,y_q]$.  On one hand, $\mathbb C[\mathbf x,\mathbf y]$ admits an action of $\GL_1(\mathbb C)$ (which is just the multiplicative group of nonzero complex numbers), via
\begin{equation*}
    (g\cdot f)(\mathbf x,\mathbf y) = f(g^{-1}\mathbf x, g\mathbf y)
\end{equation*}
for $g\in \GL_1(\mathbb C)$ and $f\in \mathbb C[\mathbf x,\mathbf y]$.  From now on let $G=\GL_1(\mathbb C)$.  Note that the invariants under the $G$-action are those polynomials in which the degree of each term is the same with respect to $\mathbf x$ as it is with respect to $\mathbf y$; in other words, $\mathbb C[\mathbf x,\mathbf y]^G$ is generated by the monomials $x_iy_j$.  (This is a special case of the First Fundamental Theorem of Invariant Theory; see \cite{gw}, Section 5.2.1.)  But since the kernel of the ring homomorphism $w_{ij} \mapsto x_iy_j$ is precisely the determinantal ideal $\mathcal I$, we have $\mathbb C[\mathbf x,\mathbf y]^G \simeq \mathbb C[\mathcal D_{p,q}^{\leq 1}]$.  (This is a special case of the Second Fundamental Theorem of Invariant Theory; see \cite{gw}, Lemma 5.2.4.)

As a result of \textit{Howe duality} in Type A --- the delicate details of which are expounded in \cite{howeremarks} and \cite{htw} --- the space $\mathbb C[\mathbf x,\mathbf y]$ is also a module under the action of the Lie algebra $\mathfrak{su}(p,q)$ by differential operators.  Upon complexification, this gives rise to an action by $\mathfrak{gl}_{p+q}(\mathbb C)$, which as a set is just the $(p+q)\times(p+q)$ complex matrices.  In particular, the invariant subring $\mathbb C[\mathbf x, \mathbf y]^G \simeq \mathbb C[\mathcal D_{p,q}^{\leq 1}]$ is the irreducible, infinite-dimensional $\mathfrak{gl}_{p+q}$-module with highest weight $(\underbrace{-1,\dots,-1,}_p 0,\dots,0)$.

This action of $\mathfrak{gl}_{p+q}$ is given by differential operators on $\mathbb C[\mathbf x,\mathbf y]$, of the following four forms (see \cite{gw}, Section 5.6):
\begin{enumerate}
\setlength\itemsep{5pt}
    \item $x_i \frac{\partial}{\partial x_j}$ (Euler operators; technically the action includes the extra term $+\delta_{ij}$);
    \item $y_i \frac{\partial}{\partial y_j}$ (Euler operators);
    \item $\frac{\partial^2}{\partial x_i \partial y_j}$ (``raising operators");
    \item $x_iy_j$ (``lowering operators").
\end{enumerate}
Note that all these operators preserve the difference between the degree with respect to $\mathbf x$ and the degree with respect to $\mathbf y$. Therefore the $\mathfrak{gl}_{p+q}$-action preserves $\mathbb C[\mathbf x, \mathbf y]^G$, which we observed is generated by the elements $x_iy_j$.  

This $\mathfrak{gl}_{p+q}$-action can be described in terms of our $\EMD_2$ setting in this paper.  First, observe that any degree-$s$ monic monomial in $\mathbb C[\mathbf x, \mathbf y]^G$ corresponds uniquely to an ordered pair of compositions $(\mu,\nu) \in \mathcal C(s,p)\times \mathcal C(s,q)$, via
\begin{equation*}
    (\mu,\nu) \longleftrightarrow \mathbf x^\mu \mathbf y^\nu := x_1^{\mu(1)} \cdots x_p^{\mu(p)} y_1^{\nu(1)} \cdots y_q^{\nu(q)}.
\end{equation*}
This is no surprise, of course, since this is just the RSK correspondence we used earlier in this subsection; written out in all of its guises, we have
\begin{equation*}
    (\mu,\nu) \longleftrightarrow \mathbf x^\mu \mathbf y^\nu \longleftrightarrow J_{(\mu,\nu)}\in \mathcal J^s_{p,q} \longleftrightarrow \prod_{i,j} w_{ij}^{J_{ij}} .
\end{equation*}

Now we can see how each type $(1)$-$(4)$ of differential operator has an interpretation in the $\EMD_2$ context.  Consider the monomial $\mathbf x^\mu \mathbf y^\nu$ as defined above.  Then, up to scaling by coefficients, we observe the following:
\begin{enumerate}
    \item The Euler operator $x_i \frac{\partial}{\partial x_j}$ corresponds to moving 1 unit in $\mu$, from bin $j$ to bin $i$, since the exponent of $x_j$ decreases by 1 and the exponent of $x_i$ increases by 1.
    \item The Euler operator $y_i \frac{\partial}{\partial y_j}$ corresponds to moving 1 unit in $\nu$, from bin $j$ to bin $i$.
    \item The raising operator $\frac{\partial^2}{\partial x_i \partial y_j}$ corresponds to removing 1 unit from each composition: from bin $i$ in $\mu$ and from bin $j$ in $\nu$.
    \item The lowering operator $x_jy_i$ corresponds to adding 1 unit to each composition: to bin $i$ in $\mu$ and to bin $j$ in $\nu$.
\end{enumerate}
It will be interesting to study further whether this connection to representation theory might be exploited in existing applications of $\EMD_2$.

\section{Proof of Proposition \ref{monge}}
\label{s:mongeproof}

The methods in this paper depended heavily upon the fact that we need consider only those arrays $J$ whose support is a chain.  This followed from the statement in Proposition \ref{monge} --- yet to be proved --- that our cost array $C$ has the Monge property.  Before proving this here, we state three useful lemmas, the first of which is proved in \cite{a&p} and \cite{park}:

\begin{lemma}
\label{submatrix}
An $n\times \cdots \times n$ array $A$ has the Monge property if and only if every two-dimensional plane of $A$ has the Monge property.
\end{lemma}

To make this explicit, we choose any two distinct indices $i,j$ from $\{1,\dots,d\}$, and then fix the remaining $d-2$ coordinates at the values $\overline{m}_1,\dots,\overline{m}_{i-1},\overline{m}_{i+1},\dots, \overline{m}_{j-1}, \overline{m}_{j+1},\dots, \overline{m}_d\in[n]$.  Then we will write $\overline{\mathbf m}^{i,j}_{k,\ell}:=(\overline{ m}_1,\dots,\overline{m}_{i-1},k,\overline{m}_{i+1},\dots, \overline{m}_{j-1}, \ell, \overline{m}_{j+1},\dots, \overline{m}_d)$. In other words, $\overline{\mathbf m}^{i,j}_{k,\ell}$ is the vector in which the $i^\text{th}$ coordinate is $k$, the $j^\text{th}$ coordinate is $\ell$, and the remaining coordinates are the fixed values $\overline{m}_1,\dots,\overline{m}_d$.  Now we can naturally define the two-dimensional subarray $A^{i,j}$ in which 
\begin{equation}
\label{subarray}
    A^{i,j}(k,\ell):=A\left(\overline{\mathbf m}^{i,j}_{k,\ell}\right).
\end{equation}  Then Lemma \ref{submatrix} states that $A$ has the Monge property if and only if $A^{i,j}$ has the Monge property for every choice of distinct $i$ and $j$.

This reduction to the two-dimensional case is extremely useful because of the following characterization of two-dimensional Monge arrays, proved in \cite{park}:

\begin{lemma}
\label{plusone}
Let $A$ be an $n\times n$ array.  Then $A$ has the Monge property if and only if \begin{equation*}
    A(k,\ell)+A(k+1,\ell+1)\leq A(k+1,\ell)+A(k,\ell+1)
    \end{equation*} 
    for all $k,\ell\in[n-1]$.
\end{lemma}
In other words, choose a position $(k,\ell)$ and then consider the $2\times 2$ subarray consisting of $A(k,\ell)$ and its three neighbors to the east, south, and southeast.  The condition displayed in the lemma means that the sum of the upper-left and lower-right entries must never be greater than the sum of the lower-left and upper-right entries.

We will need one final lemma, specific to the cost function $C$ in this paper.  Recall from Proposition \ref{costalt} that if we let $\widetilde{\mathbf m}$ denote a vector $\mathbf m$ with its coordinates rearranged in ascending order, then \begin{equation*}
    C(\mathbf m)=-\widetilde m_1 - \cdots -\widetilde m_{\lfloor \frac{d+1}{2} \rfloor} + \widetilde m_{\lfloor \frac{d+1}{2} \rfloor+1} + \cdots +\widetilde m_d \hspace{1cm} (d \text{ even})
\end{equation*}
or
\begin{equation*}
    C(\mathbf m)=-\widetilde m_1 - \cdots -\widetilde m_{\lfloor \frac{d+1}{2} \rfloor-1} + \widetilde m_{\lfloor \frac{d+1}{2} \rfloor+1} + \cdots + \widetilde m_d \hspace{1cm} (d \text{ odd}).
\end{equation*}
The index $\lfloor \frac{d+1}{2} \rfloor$ gave a kind of ``median" of the coordinates in $\mathbf m$; from now on, however, we will work instead with $M:=\lfloor \frac{d+1}{2} \rfloor+1=\lceil \frac{d+2}{2}\rceil$.  Intuitively, this index $M$ gives the next-greatest coordinate after the ``median."  The picture is the following, where the vertical lines divide the coordinates into two equal sets (with one leftover coordinate in the middle if $d$ is odd:
\begin{align*}
    d \text{ even}: \hspace{1cm}\widetilde{\mathbf m}&=(\widetilde m_1, \dots, \phantom{\Bigg\lvert}\widetilde m_{M-1}, \Bigg\lvert \widetilde m_M, \dots, \widetilde m_d)\\
    d \text{ odd}: \hspace{1cm}\widetilde{\mathbf m}&=(\widetilde m_1, \dots,\Bigg\lvert \widetilde m_{M-1}, \Bigg\lvert \widetilde m_M, \dots, \widetilde m_d)
\end{align*}
With this indexing in mind, we state our final lemma, which records the effect on $C(\mathbf m)$ of adding $1$ to a single coordinate $m_i$.  Recall from earlier that $\mathbf e(i)$ denotes the vector whose coordinates are all $0$ except for a $1$ in the $i^\text{th}$ component.

\begin{lemma}
\label{incdecsame}
Adding $1$ to a single coordinate $m_i$ of $\mathbf m$ has one of three effects on $C(\mathbf m)$: it either increases by 1, decreases by 1, or remains the same.  The effect depends on the value of $m_i$ relative to the other coordinates of $\mathbf m$:

\begin{enumerate}
    \item \normalfont $C(\mathbf m +\mathbf e(i))=C(\mathbf m)+1$ if $m_i \geq \widetilde m_M$.
    \item $C(\mathbf m +\mathbf e(i))=C(\mathbf m)-1$ if:
    \begin{enumerate}
        \item $d$ is even and $m_i<\widetilde m_M$; or
        \item $d$ is odd and $m_i<\widetilde m_{M-1}$.
    \end{enumerate}
    \item $C(\mathbf m +\mathbf e(i))=C(\mathbf m)$ if $d$ is odd and $m_i=\widetilde m_{M-1} < \widetilde m_M$.
\end{enumerate}
\end{lemma}

\begin{proof} We prove each of the three cases; the reader may find it helpful to keep an eye on the two possible ``pictures" of $\widetilde{\mathbf m}$ displayed before this lemma, along with the two possible sums for $C(\mathbf m)$ displayed just before that.
\begin{enumerate}
    \item Assume $m_i \geq \widetilde m_M$.  Then $m_i+1>\widetilde m_M$, and so in the sum defining $C(\mathbf m)$, we must have positive $m_i$ replaced by positive $(m_i+1)$.  Hence $C(\mathbf m)$ has increased by 1.
    \item \begin{enumerate}
        \item Assume $d$ is even and $m_i < \widetilde m_M$.  Then $m_i+1 \leq \widetilde m_M$, and so in the sum defining $C(\mathbf m)$, we must have negative $m_i$ replaced by negative $(m_i+1)$.  Hence $C(\mathbf m)$ has decreased by 1.
        \item Assume $d$ is odd and $m_i < \widetilde m_{M-1}$.  Then $m_i+1 \leq \widetilde m_{M-1}$, and so we must have negative $m_i$ replaced by negative $(m_i+1)$.  Hence $C(\mathbf m)$ has decreased by 1.
    \end{enumerate}
    \item Assume $d$ is odd and $m_i=\widetilde m_{M-1}<\widetilde m_M$; note that $\widetilde m_{M-1}$ does not appear in the sum defining $C(\mathbf m)$.  Then $\widetilde m_{M-2}<m_i+1\leq \widetilde m_M$, and so $m_i+1$ still does not appear in the sum defining $C(\mathbf m+\mathbf e(i))$.  Hence $C(\mathbf m)$ remains unchanged.
\end{enumerate}
\end{proof}

We are now ready for the proof, in which we show that an arbitrary two-dimensional subarray of $C$ has the Monge property.

\begin{proof}[Proof of Proposition \ref{monge}]
Let $i,j$ be two distinct indices in $\{1, \dots, d\}$.  Fix the remaining coordinates $\overline m_1,\dots,\overline m_d$ as above, and let $C^{i,j}$ be the corresponding two-dimensional subarray of $C$ defined in \eqref{subarray}.  Now let $m_i, m_j \in [n-1]$.  By Lemmas \ref{submatrix} and \ref{plusone}, it will suffice to show that \begin{equation*}
    C^{i,j}(m_i,m_j)+C^{i,j}(m_i+1,m_j+1)\leq C^{i,j}(m_i+1,m_j)+C^{i,j}(m_i,m_j+1).
\end{equation*}
But this condition can be rewritten as the following, where we simply write $\overline{\mathbf m}$ for $\overline{\mathbf m}^{i,j}_{m_i,m_j}$:
\begin{equation}
\label{condition}
    C(\overline{\mathbf m})+C(\overline{\mathbf m}+\mathbf e(i)+\mathbf e(j))\leq C(\overline{\mathbf m}+\mathbf e(i)) + C(\overline{\mathbf m}+\mathbf e(j))
\end{equation}
To show that this condition holds true, we need to examine six possible cases, depending on whether adding 1 to $m_i$ and $m_j$ (independently) causes $C$ to increase, decrease, or remain the same:
\bigskip

\small \begin{tabular}{|c||c|c|c|}\hline
     & $C(\overline{\mathbf m}+\mathbf e(i))=C(\overline{\mathbf m})+1$ & $C(\overline{\mathbf m}+\mathbf e(i))=C(\overline{\mathbf m})-1$  & $C(\overline{\mathbf m}+\mathbf e(i))=C(\overline{\mathbf m})$ \\ \hline \hline
     $C(\overline{\mathbf m}+\mathbf e(j))=C(\overline{\mathbf m})+1$ & Case 1 &  &  \\ \hline
     $C(\overline{\mathbf m}+\mathbf e(j))=C(\overline{\mathbf m})-1$ & Case 2 & Case 4 &  \\ \hline
     $C(\overline{\mathbf m}+\mathbf e(j))=C(\overline{\mathbf m})$ & Case 3 & Case 5 & Case 6 \\
     \hline
\end{tabular}

\bigskip

\normalsize
In each case below, all simplifications are directly justified by the results in Lemma \ref{incdecsame}. 

\begin{itemize} 
\item \textbf{Case 1:} In this case, the right-hand side of \eqref{condition} is $2\cdot C(\overline{\mathbf m})+2$.  For the left-hand side, we know in general that $C\big(\overline{\mathbf m}+\mathbf e(i)+\mathbf e(j)\big)=C\big((\overline{\mathbf m}+\mathbf e(i))+\mathbf e(j)\big)$, which by Lemma \ref{incdecsame} can be no greater than $C(\overline{\mathbf m})+2$.  Hence the inequality in \eqref{condition} must hold.

\item \textbf{Case 2:} In this case, the right-hand side of \eqref{condition} is $2\cdot C(\overline{\mathbf m})$.  As for the second term on the left-hand side, by Lemma \ref{incdecsame}, we must have $m_i \geq \widetilde m_M$; meanwhile, $m_j$ is strictly less than either $\widetilde m_M$ (if $d$ is even) or $\widetilde m_{M-1}$ (if $d$ is odd), and so neither inequality is affected by adding 1 to $m_i$.  Therefore we have 
\begin{align*}
    C\big(\overline{\mathbf m}+\mathbf e(i)+\mathbf e(j)\big)&=C\big((\overline{\mathbf m}+\mathbf e(i))+\mathbf e(j)\big)\\
    &=C((\overline{\mathbf m}+\mathbf e(i))-1\\
    &= C(\overline{\mathbf m})+1-1\\
    &=C(\overline{\mathbf m}).
\end{align*}
Hence we have an equality in \eqref{condition}.

\item \textbf{Case 3:} Similar to Case 2, the two additions are independent of each other.  The right-hand side of \eqref{condition} is $2\cdot C(\overline{\mathbf m})+1$.  In this case, we must have $d$ odd; also, $m_i\geq \widetilde m_M$, along with $m_j=\widetilde m_{M-1}<\widetilde m_M$.  Then \begin{align*}
    C\big(\overline{\mathbf m}+\mathbf e(i)+\mathbf e(j)\big)&=C\big((\overline{\mathbf m}+\mathbf e(i))+\mathbf e(j)\big)\\
    &=C((\overline{\mathbf m}+\mathbf e(i))\\
    &= C(\overline{\mathbf m})+1.
\end{align*} 
Again we obtain an equality in \eqref{condition}.

\item \textbf{Case 4:} The right-hand side of \eqref{condition} is $2\cdot C(\overline{\mathbf m})-2$.  If $d$ is even, then both $m_i$ and $m_j$ are strictly less than $\widetilde m_M$, and if $d$ is odd, then both are strictly less than $\widetilde m_{M-1}$.  Either way, after adding 1 to $m_i$, the same inequality still holds for $m_j$, and so again we have
\begin{align*}
    C\big(\overline{\mathbf m}+\mathbf e(i)+\mathbf e(j)\big)&=C\big((\overline{\mathbf m}+\mathbf e(i))+\mathbf e(j)\big)\\
    &=C((\overline{\mathbf m}+\mathbf e(i))-1\\
    &= C(\overline{\mathbf m})-1-1\\
    &=C(\overline{\mathbf m})-2,
\end{align*}
and we get an equality in \eqref{condition}.

\item \textbf{Case 5:} The right-hand side of \eqref{condition} is $2\cdot C(\overline{\mathbf m})-1$.  In this case, $d$ must be odd, with $m_i<\widetilde m_{M-1} = m_j <\widetilde m_M$.  After adding 1 to $m_j$, we still have $m_i$ less than the $(M-1)^\text{th}$ component in the new rearranged vector, and so the effects of the two additions are independent.  We obtain
\begin{align*}
    C\big(\overline{\mathbf m}+\mathbf e(i)+\mathbf e(j)\big)&=C\big((\overline{\mathbf m}+\mathbf e(j))+\mathbf e(i)\big)\\
    &=C((\overline{\mathbf m}+\mathbf e(j))-1\\
    &= C(\overline{\mathbf m})-1
\end{align*}
and so we have an equality in \eqref{condition}.

\item \textbf{Case 6:} This is the slightly surprising case, in which the two additions are \emph{not} independent of each other.  The right-hand side of \eqref{condition} is $2\cdot C(\overline{\mathbf m})$, and we know that $d$ must be odd, with $m_i=m_j=\widetilde m_{M-1}<\widetilde m_M$.  After adding 1 to $m_i$, we obtain a vector $\mathbf m'$ in which $m'_j=m_j$ is now strictly less than $\widetilde m'_{M-1}$, and so \emph{now} adding 1 to $m_j$ results in an overall decrease by 1.  Hence we have
\begin{align*}
    C\big(\overline{\mathbf m}+\mathbf e(i)+\mathbf e(j)\big)&=C\big((\overline{\mathbf m}+\mathbf e(i))+\mathbf e(j)\big)\\
    &=C((\overline{\mathbf m}+\mathbf e(i))-1\\
    &= C(\overline{\mathbf m})-1.
\end{align*}
Hence the left-hand side of \eqref{condition} is less than the right-hand side, and the condition is still satisfied.
\end{itemize}

We have exhausted all possible cases, and so since \eqref{condition} holds in each of them, the two-dimensional array $C^{i,j}$ has the Monge property.  Since $i$ and $j$ were arbitrary, \emph{every} two-dimensional subarray of $C$ has the Monge property, and so by Lemma \ref{submatrix}, we conclude that $C$ itself has the Monge property.
\end{proof} 

\bibliographystyle{amsplain}

\bibliography{emdbib}

\end{document}